\documentclass[a4paper, 12pt, twoside, notitlepage]{amsart}
\usepackage{amsmath,amscd}
\usepackage{amssymb}
\usepackage{amsthm}
\usepackage{comment}
\usepackage{graphicx, xcolor}
\usepackage{lmodern}
\usepackage{microtype}
\usepackage{nowidow}

\usepackage{mathrsfs}
\usepackage[ocgcolorlinks, linkcolor=blue]{hyperref}

\usepackage{bm}
\usepackage{bbm}
\usepackage{url}

\usepackage[utf8]{inputenc}
\usepackage{mathtools,amssymb}
\usepackage{esint}
\usepackage{tikz}
\usepackage{dsfont}
\usepackage{relsize}
\usepackage{url}
\usepackage{xcolor}
\usepackage{graphicx}
\usepackage{mathrsfs}
\usepackage[shortlabels]{enumitem}
\usepackage{lineno}
\usepackage{amsmath}
\usepackage{enumitem}
\usepackage{amsthm} 
\usepackage{verbatim}
\usepackage{dsfont}
\usepackage{faktor}
\numberwithin{equation}{section}

\allowdisplaybreaks

\mathtoolsset{showonlyrefs}

\graphicspath{{images/}}

\newtheorem{theorem}{Theorem}[section]
\newtheorem{lemma}[theorem]{Lemma}

\newtheorem{proposition}[theorem]{Proposition}
\newtheorem{question}{Question}
\newtheorem{remark}[theorem]{Remark}

\title[Nonlocal diffuse optical tomography]{Inverse problem for a nonlocal diffuse optical tomography equation}

\author[P. Zimmermann]{Philipp Zimmermann}
\address{Department of Mathematics, ETH Zurich, Z\"urich, Switzerland}
\email{philipp.zimmermann@math.ethz.ch}

\newcommand{\R}{{\mathbb R}}

\newcommand{\N}{{\mathbb N}}

\newcommand{\schwartz}{\mathscr{S}}

\newcommand{\tempered}{\mathscr{S}^{\prime}}

\newcommand{\fourier}{\mathcal{F}}
\newcommand{\ifourier}{\mathcal{F}^{-1}}
\newcommand{\vev}[1]{\left\langle#1\right\rangle}

\newcommand{\distr}{\mathscr{D}^{\prime}}
\DeclarePairedDelimiter{\abs}{\lvert}{\rvert}
\DeclarePairedDelimiter{\norm}{\|}{\|}
\newcommand{\ip}[2]{\left\langle #1,#2 \right\rangle}
\DeclareMathOperator{\supp}{supp} 
\DeclareMathOperator{\dist}{dist} 

\begin{document}

	\maketitle
	\begin{abstract}
	In this article a nonlocal analogue of an inverse problem in diffuse optical tomography is considered. We show that whenever one has given two pairs of diffusion and absorption coefficients $(\gamma_j,q_j)$, $j=1,2$, such that there holds $q_1=q_2$ in the measurement set $W$ and they generate the same DN data, then they are necessarily equal in $\R^n$ and $\Omega$, respectively. Additionally, we show that the condition $q_1|_W=q_2|_W$ is optimal in the sense that without this restriction one can construct two distinct pairs $(\gamma_j,q_j)$, $j=1,2$ generating the same DN data.

		\medskip
		
		

	\end{abstract}

	\tableofcontents

	\section{Introduction}
    \label{sec: introduction}

    In recent years many different nonlocal inverse problems have been studied. The prototypical example is the inverse problem for the fractional Schr\"odinger operator $(-\Delta)^s+q$, where the measurements are encoded in the (exterior) Dirichlet to Neumann (DN) map $f\mapsto\Lambda_qf=(-\Delta)^su_f|_{\Omega_e}$. Here $\Omega_e=\R^n\setminus\overline{\Omega}$ is the exterior of a smoothly bounded domain $\Omega\subset\R^n$ and $0<s<1$. This problem, nowadays called \emph{fractional Calder\'on problem}, was first considered for $q\in L^{\infty}(\Omega)$ in \cite{GSU20} and initiated many of the later developments. The classical proof of the (interior) uniqueness for the fractional Calder\'on problem, that is of the assertion that $\Lambda_{q_1}=\Lambda_{q_2}$ implies $q_1=q_2$ in $\Omega$, relies on the Alessandrini identity, the unique continuation principle (UCP) of the fractional Laplacian and the Runge approximation. Following a similar approach, in the works \cite{bhattacharyya2021inverse,CMR20,CMRU20,GLX,CL2019determining,CLL2017simultaneously,cekic2020calderon,feizmohammadi2021fractional,harrach2017nonlocal-monotonicity,harrach2020monotonicity,GRSU18,GU2021calder,ghosh2021non,lin2020monotonicity,LL2020inverse,LL2022inverse,LLR2019calder,LLU2022calder,KLW2021calder,RS17,ruland2018exponential,RZ2022unboundedFracCald}, it has been shown that one can uniquely recover lower order, local perturbations of many different nonlocal models.
    
    On the other hand, the author together with different collaborators considered in \cite{RZ2022unboundedFracCald,counterexamples,RGZ2022GlobalUniqueness,RZ2022LowReg,StabilityFracCond} the \emph{inverse fractional conductivity problem}, which has been first studied in \cite{covi}. The main objective in this problem is to uniquely determine the conductivity $\gamma\colon\R^n\to\R_+$ from the DN map $f\mapsto \Lambda_{\gamma}f$ related to the Dirichlet problem
    \begin{equation}
    \label{eq: frac cond eq intro}
    \begin{split}
            L^s_\gamma u &= 0 \enspace\text{in}\enspace\Omega,\\
            u &= f \enspace\text{in}\enspace\Omega_e.
    \end{split}
    \end{equation}
    Here $L_{\gamma}^s$ denotes the \emph{fractional conductivity operator}, which can be strongly defined via
    \begin{equation}
    \label{eq: frac cond operator intro}
        L_{\gamma}^su(x)=C_{n,s}\gamma^{1/2}(x)\,\text{p.v.}\int_{\R^n}\gamma^{1/2}(y)\frac{u(x)-u(y)}{|x-y|^{n+2s}}\,dy.
    \end{equation}
    In this formula, $C_{n,s}>0$ is some positive constant and $\text{p.v.}$ denotes the Cauchy principal value. More concretely, in the aforementioned articles it has been shown that the conductivity $\gamma$ with background deviation $m_{\gamma}=\gamma^{1/2}-1$ in $H^{s,n/s}(\R^n)$ can be uniquely recovered from the DN data, in the measurement set the conductivity can be explicitly reconstructed with a Lipschitz modulus of continuity and on smooth, bounded domains the full data inverse fractional conductivity problem is under suitable a priori assumptions logarithmically stable.
    
    Let us note that as $s$ converges to 1, the fractional conductivity operator $L_{\gamma}^s$ becomes the \emph{conductivity operator} $L_{\gamma}u=-\text{div}(\gamma\nabla u)$. Hence the above inverse problem can be considered as a nonlocal analogue of the classical \emph{Calder\'on problem} \cite{calderon}, that is, the problem of uniquely recovering the conductivity $\gamma\colon\overline{\Omega}\to\R_+$ from the DN map $f\mapsto\Lambda_{\gamma}f=\gamma\partial_{\nu}u_f|_{\partial\Omega}$, where $u_f\in H^1(\Omega)$ is the unique solution to the Dirichlet problem of the \emph{conductivity equation}
    \begin{equation}
    \label{calderon prob intro}
    \begin{split}
            L_\gamma u &= 0 \,\enspace\text{in}\enspace\Omega,\\
            u &= f \enspace\text{on}\enspace\partial\Omega
    \end{split}
    \end{equation}
    and $\nu$ denotes the outward pointing unit normal vector of the smoothly bounded domain $\Omega\subset\R^n$. The mathematical investigation of the inverse conductivity problem dates at least back to the work \cite{Langer33-calderon-halfspace} of Langer. Many uniqueness proofs of the Calderón problem are based on the \emph{Liouville reduction}, which allows to reduce this inverse problem for a variable coefficient operator to the inverse problem for the Schrödinger equation $-\Delta+q$, on the construction of complex geometric optics (CGO) solutions \cite{SU87}, and on a boundary determination result \cite{KV84}. The first uniqueness proof for the inverse fractional conductivity problem also relied on a reduction of the problem via a \emph{fractional Liouville reduction} to the inverse problem for the fractional Schr\"odinger equation and the boundary determination of Kohn and Vogelius was replaced by an exterior determination result (cf.~\cite{RGZ2022GlobalUniqueness} for the case $m_{\gamma}\in H^{2s,n/2s}(\R^n)$ and \cite{RZ2022LowReg} for $m_{\gamma}\in H^{s,n/s}(\R^n)$). Since the UCP and the Runge approximation are much stronger for nonlocal operators than for local ones, which in turn relies on the fact solutions to $(-\Delta)^s+q$ are much less rigid than the ones to the local Schr\"odinger equation $-\Delta+q$, the uniqueness for the nonlocal Schr\"odinger equation can be established without the construction of CGO solutions. In fact, it is an open problem whether these exist for the fractional Schr\"odinger equation. 
    
    \subsection{The optical tomography equation}

    Recently, in the articles \cite{OptTomHarrach,SimDetDiffAbs}, it has been investigated whether the diffusion $\gamma$ and the absorption coefficient $q$ in the \emph{optical tomography equation}
    \begin{equation}
    \label{eq: opt tom prob}
    \begin{split}
            L_{\gamma}u +qu &= F \,\enspace\text{in}\enspace\Omega
    \end{split}
    \end{equation}
    can be uniquely recovered from the partial Cauchy data $(u|_{\Gamma},\gamma\partial_{\nu}u|_{\Gamma})$, where $\Omega\subset\R^n$ is a bounded domain and $\Gamma \subset\partial\Omega$ is an arbitrarily small region of the boundary. This problem arises in the (stationary) diffusion based optical tomography and therefore we refer to \eqref{eq: opt tom prob} as the optical tomography equation. Generally speaking, in optical tomography one uses low energy visible or near infrared light (wavelength $\lambda\sim 700-1000$nm) to test highly scattering media (as a tissue sample of a human body) and wants to reconstruct the optical properties within the sample by intensity measurements on the boundary. In a possible experimental situation, light is sent via optical fibres to the surface of the medium under investigation and the transilluminated light is measured by some detecting fibres. 
    
    The starting point to describe the radiation propagation in highly scattering media is the radiative transfer equation (Boltzmann equation)
    \begin{equation}
    \label{eq: radiation equation}
        \begin{split}
            &\partial_tI(x,t,v)+v\cdot\nabla I(x,t,v)+(\mu_a+\mu_s)I(x,t,v)\\
            &\quad=\mu_s\int_{S^{n-1}}f(v,v')I(x,t,v')\,d\sigma(v')+G(x,t,v),
        \end{split}
    \end{equation}
    which describes the change of the radiance $I=I(x,t,v)$ at spacetime point $(x,t)$ into the direction $v\in S^{n-1}=\{x\,;\, |x|=1\}$. Here, we set $c=1$ (speed of light) and the other quantities have the following physical meaning:\\
    
    \begin{tabular}{ c l }
 $\mu_a$ & absorption coefficient \\ 
 $\mu_s$ & scattering coefficient  \\  
 $f(v,v')$ & scattering phase function - probability that the wave \\
 & incident in direction $v'$ is scattered into direction $v$ \\
 $G$ & isotropic source
\end{tabular}\\
    
    In the diffusion approximation, as explained in detail in \cite{OptTomArridge} or \cite[Appendix]{schweiger1995finite}, one gets equation \eqref{eq: opt tom prob}, where the quantities are related as follows:\\
    
     \begin{tabular}{ c l }
 $u$ & photon density - $u(x,t)=\int_{S^{n-1}}I(x,t,v')\,d\sigma(v')$ \\ 
 $\gamma$ & diffusion coefficient of the medium -  $\gamma=[3(\mu_a+\mu'_s)]^{-1}$ \\
 & with $\mu'_s$ being the reduced scattering coefficient \\  
 $q$ & absorption coefficient $\mu_a$  \\
 $F$ & isotropic source
\end{tabular}\\

and
 $-\gamma\partial_{\nu}u|_{\Gamma}$ describes the normal photon current (or exitance) across $\Gamma\subset\partial\Omega$. Let us remark that in the diffusion approximation one assumes $\mu_a\ll \mu_s$ and that the light propagation is weakly anisotropic, which is incoorporated in $\mu'_s$. For further discussion on this classical model, we refer to the above cited articles and \cite{gibson2005recent}.

    \subsubsection{Non-uniqueness in diffusion based optical tomography}
    \label{subsubsec: nonunique OT}
    
    In \cite{arridge1998nonuniqueness}, Arridge and Lionheart constructed counterexamples to uniqueness for the inverse problem of the diffusion based optical tomography equation \eqref{eq: opt tom prob}. They consider a smoothly bounded domain $\Omega\subset\R^n$ containing a compact subdomain $\Omega_0\Subset \Omega$ such that the isotropic source is supported in $\Omega_1\vcentcolon = \Omega\setminus\overline{\Omega}_0$. Then they observe that if the diffusion coefficient $\gamma$ is sufficiently regular, the optical tomography equation \eqref{eq: opt tom prob} is reduced via the Liouville reduction to 
    \begin{equation}
    \label{eq: reduced opt tom prob}
    \begin{split}
            -\Delta v +\eta v &= \frac{F}{\gamma^{1/2}} \,\enspace\text{in}\enspace\Omega
    \end{split}
    \quad\text{with}\quad \eta\vcentcolon = \frac{\Delta\gamma^{1/2}}{\gamma^{1/2}}+\frac{q}{\gamma},
    \end{equation}
    where $v=\gamma^{1/2}u$. Now, one can change the coefficients $(\gamma,q)$ to
    \begin{equation}
    \label{eq: equivalent coeff}
        \widetilde{\gamma}\vcentcolon = \gamma+\gamma_0,\quad \widetilde{q}\vcentcolon = q+q_0\quad\text{and}\quad\widetilde{\eta}\vcentcolon=\frac{\Delta\widetilde{\gamma}^{1/2}}{\widetilde{\gamma}^{1/2}}+\frac{\widetilde{q}}{\widetilde{\gamma}},
    \end{equation}
    where these new parameters satisfy
    \begin{enumerate}[(i)]
        \item\label{cond 1 non} $\gamma_0\geq 0$ with $\gamma_0|_{\Omega_1}=0$
        \item\label{cond 2 non} and $\widetilde{\eta}=\eta$ in $\Omega$.
    \end{enumerate}
    The latter condition means nothing else than
    \begin{equation}
    \label{eq: effective potentials}
        \frac{\Delta(\gamma+\gamma_0)^{1/2}}{(\gamma+\gamma_0)^{1/2}}+\frac{q+q_0}{\gamma+\gamma_0}=\frac{\Delta\gamma^{1/2}}{\gamma^{1/2}}+\frac{q}{\gamma}\quad\text{in}\quad\Omega.
    \end{equation}
    Hence, if we have given $\gamma_0$, then this relation can always be used to calculate $q_0$ by
    \begin{equation}
    \label{eq: calculation of potential perturb}
        q_0=(\gamma+\gamma_0)\left(\frac{\Delta\gamma^{1/2}}{\gamma^{1/2}}-\frac{\Delta(\gamma+\gamma_0)^{1/2}}{(\gamma+\gamma_0)^{1/2}}+\frac{q}{\gamma}\right)-q.
    \end{equation}
    As the transformations \eqref{eq: equivalent coeff} under the conditions \ref{cond 1 non}, \ref{cond 2 non} leave the Dirichlet and Neumann data of solutions to \eqref{eq: reduced opt tom prob} invariant, this leads to the desired counterexamples.
    
    \subsubsection{Uniqueness in diffusion based optical tomography}
    
    Harrach considered in \cite{OptTomHarrach,SimDetDiffAbs} the discrepancy between the counterexamples of the last section and the positive experimental results in \cite[Section~3.4.3]{gibson2005recent} of recovering $\gamma$ and $q$ simultaneously in more detail. In these works it is established that uniqueness in the inverse problem for the optical tomography equation is obtained, when the diffusion $\gamma$ is piecewise constant and the absorption coefficient piecewise analytic. The main tool to obtain this result is the technique of localized potentials (see~\cite{LocPotential}), which are solutions of \eqref{eq: opt tom prob} that are large on a particular subset but otherwise small. The use of special singular solutions to prove uniqueness in inverse problems for (local or nonlocal) PDEs became in recent years a popular technique (see for example \cite{KV84,KV85,Alessandrini-singular,Nachman1996GlobalUniqueness,SU87} for local PDEs and \cite{RGZ2022GlobalUniqueness,RZ2022LowReg,LRZ22,KLZ22FracpLap} for nonlocal PDEs).
    
    \subsection{Nonlocal optical tomography equation and main results}
   
   The main goal of this article is to study a nonlocal variant of the previously introduced inverse problem for the optical tomography equation. More concretely, we consider the \emph{nonlocal optical tomography equation}
   \begin{equation}
   \label{eq: nonlocal tomography equation intro}
       L_{\gamma}^su+qu=0\quad\text{in}\quad\Omega,
   \end{equation}
   where $\Omega\subset\R^n$ is a domain bounded in one direction, $0<s<1$, $\gamma\colon \R^n\to\R_+$ is a diffusion coefficient, $q\colon\R^n\to\R$ an absorption coefficient (aka potential) and $L_{\gamma}^s$ the variable coefficient nonlocal operator defined in \eqref{eq: frac cond operator intro}. Then we ask:
  
   \begin{question}
   \label{question uniqueness}
     Let $W_1,W_2\subset \Omega_e$ be two measurement sets. Under what conditions does the DN map $C_c^{\infty}(W_1)\ni f\mapsto \Lambda_{\gamma,q}f|_{W_2}$ related to \eqref{eq: nonlocal tomography equation intro} uniquely determine the coefficients $\gamma$ and $q$?
   \end{question}
   
    By \cite[Theorem~1.8]{RZ2022LowReg}, we know that the measurement sets need to satisfy $W_1\cap W_2\neq\emptyset$ and hence, we consider the setup illustrated in Figure~\ref{figure 1}.
 \begin{figure}[!ht]
    \centering
    \begin{tikzpicture}
    \filldraw[color=blue!50, fill=blue!5, xshift=11cm, yshift=1.5cm] (3,-2.5) ellipse (0.9 and 0.9);    
    \node[xshift=11cm, yshift=1.5cm] at (3,-2.5) {$\raisebox{-.35\baselineskip}{\ensuremath{\Lambda_{\gamma,q}f|_{W_2}}}$};
    \filldraw[color=green!50, fill=green!5, xshift=11cm, yshift=0.1cm, opacity=0.8] (3,-2.5) ellipse (1.3 and 0.75);    
    \node[xshift=11cm, yshift=0.1cm] at (3,-2.5) {$\raisebox{-.35\baselineskip}{\ensuremath{f\in C_c^{\infty}(W_1)}}$};
    \filldraw [color=orange!80, fill = orange!5, xshift=8cm, yshift=-2cm,opacity=0.8] plot [smooth cycle] coordinates {(-1,0.9) (3,1.5) (4.5,-0.5) (3,-1) (-1.5,-0.25)};
    \node[xshift=3cm] at (6.3,-1.75) {$\raisebox{-.35\baselineskip}{\ensuremath{L_{\gamma}^su+qu=0\enspace\text{in}\enspace\Omega}}$};
\end{tikzpicture}
\caption{\begin{small} Here, $\Omega$ represents the scattering medium, $\gamma$, $q$ the diffusion and absorption coefficient, $f$ a light pulse in $W_1$ and $\Lambda_{\gamma}f|_{W_2}$ the nonlocal photon current in $W_2$.\end{small}}
\label{figure 1}
\end{figure}
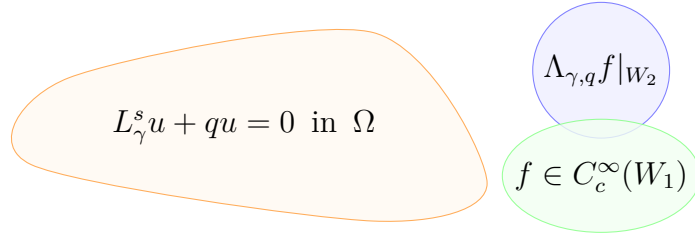
   
   Moreover, motivated by the counterexamples in Section~ \ref{subsubsec: nonunique OT}, we expect that the potentials $q_1,q_2$ should coincide in the measurement sets $W_1,W_2\subset\Omega_e$. Indeed, under slightly weaker assumptions we establish that the DN map $\Lambda_{\gamma,q}$ uniquely determines the coefficients $\gamma$ and $q$. More precisely, we will prove in Section~\ref{sec: inverse problem} the following result:
   
   \begin{theorem}[Global uniqueness]
    \label{main thm}
        Let $0 < s < \min(1,n/2)$, suppose $\Omega\subset \R^n$ is a domain bounded in one direction and let $W_1,W_2\subset \Omega_e$ be two non-disjoint measurement sets. Assume that the diffusions $\gamma_1, \gamma_2\in L^{\infty}(\R^n)$ with background deviations $m_{\gamma_1},m_{\gamma_2}\in H^{s,n/s}(\R^n)$ and potentials $q_1,q_2\in \distr(\R^n)$ satisfy
        \begin{enumerate}[(i)]
            \item\label{uniform ellipticity diffusions} $\gamma_1,\gamma_2$ are uniformly elliptic with lower bound $\gamma_0>0$,
            \item\label{continuity diffusions} $\gamma_1, \gamma_2$ are a.e. continuous in $W_1\cap W_2$,
            \item\label{integrability potentials} $q_1,q_2\in M_{\gamma_0/\delta_0,+}(H^s\to H^{-s})\cap L^p_{loc}(W_1\cap W_2)$ for some $\frac{n}{2s}< p\leq \infty$
            \item\label{equal potentials in measurement sets} and $q_1|_{W_1\cap W_2}=q_2|_{W_1\cap W_2}$.
        \end{enumerate}
        If $\Lambda_{\gamma_1,q_1}f|_{W_2}=\Lambda_{\gamma_2,q_2}f|_{W_2}$ for all $f\in C_c^{\infty}(W_1)$, then there holds $\gamma_1=\gamma_2$ in $\R^n$ and $q_1=q_2$ in $\Omega$.
    \end{theorem}
    
    \begin{remark}
    \label{remark: def of constant}
    In the above theorem and throughout this article, we set $\delta_0\vcentcolon =2\max(1,C_{opt})$, where $C_{opt}=C_{opt}(n,s,\Omega)>0$ is the optimal fractional Poincar\'e constant defined via
    \begin{equation}
    \label{eq: optimal fractional Poincare constant}
        C_{opt}^{-1}=\inf_{0\neq u\in\widetilde{H}^s(\Omega)}\frac{[u]^2_{H^s(\R^n)}}{\|u\|_{L^2(\R^n)}^2}<\infty
    \end{equation}
    (see~Theorem~\ref{thm: Poinc Unbounded Doms}).
\end{remark}

    \begin{remark}
     Let us note that when we change $q$ away from $\Omega$ and the measurement sets $W_1,W_2$, then the DN data $C_c^{\infty}(W_1)\ni f\mapsto \Lambda_{\gamma,q}f|_{W_2}$ remain the same. Therefore, in the above theorem we have only uniqueness for the potential in $\Omega$.
   \end{remark}
    
    Next, let us discuss the assumption that the potentials $q_1,q_2$ coincide in $W=W_1\cap W_2$, where $W_1,W_2\subset\Omega_e$ are two non-disjoint measurement sets. First of all, one can observe that the proofs given in Section~\ref{subsec: exterior reconstruction} and \ref{subsec: determination of diffusion coeff} still work under the seemingly weaker assumption $W\cap \text{int}(\{q_1=q_2\})\neq\emptyset$. Hence, one can again conclude that $\gamma_1=\gamma_2$ in $\R^n$. Now, the UCP of the fractional conductivity operator $L_{\gamma}^s$ (see~Theorem~\ref{thm: uniqueness q}) and \cite[Corollary~2.7]{RZ2022unboundedFracCald} show that $q_1=q_2$ in $W$. Therefore, if the DN maps coincide then the assumption $W\cap \text{int}(\{q_1=q_2\})\neq\emptyset$ is equally strong as $q_1=q_2$ in $W$. This leads us to the following question:
     
     \begin{question}
   \label{question non-uniqueness}
        For a measurement set $W\subset\Omega_e$, can one find two distinct pairs of diffusion and absorption coefficients $(\gamma_1,q_1),(\gamma_2,q_2)$ satisfying the conditions \ref{uniform ellipticity diffusions}-\ref{integrability potentials} in Theorem~\ref{main thm} that generate the same DN data, i.e. $\Lambda_{\gamma_1,q_1}f|_{W}=\Lambda_{\gamma_2,q_2}f|_{W}$ for all $f\in C_c^{\infty}(W)$, but $q_1\not\equiv q_2$ in $W$?
   \end{question}

   We establish the following result:
   
   \begin{theorem}[Non-uniqueness]
   \label{thm: non uniqueness}
        Let $0 < s < \min(1,n/2)$, suppose $\Omega\subset \R^n$ is a domain bounded in one direction and let $W\subset \Omega$ be a measurement set. Then there exist two different pairs $(\gamma_1,q_1)$ and $(\gamma_2,q_2)$ satisfying $\gamma_1,\gamma_2\in L^{\infty}(\R^n)$, $m_{\gamma_1},m_{\gamma_1}\in H^{s,n/s}(\R^n)$, \ref{uniform ellipticity diffusions}--\ref{integrability potentials} of Theorem~\ref{main thm} and $\left.\Lambda_{\gamma_1,q_1}f\right|_{W}=\left.\Lambda_{\gamma_2,q_2}f\right|_{W}$ for all $f\in C_c^{\infty}(W)$, but there holds $q_1(x)\neq q_2(x)$ for all $x\in W$.
   \end{theorem}
   
   Finally, let us note that whether uniqueness or non-uniqueness holds in the general case $q_1\not\equiv q_2$ on $W$ but $W\cap \{q_1=q_2\}$ has no interior points, is not answered by the above results. In fact, if $q_1,q_2$ are arbitrary potentials and the assumption $\Lambda_{\gamma_1,q_2}f|_{W}=\Lambda_{\gamma_2,q_2}f|_{W}$ for all $f\in C_c^{\infty}(W)$ implies $\gamma_1=\gamma_2$ in $\R^n$, then \cite[Corollary~2.7]{RZ2022unboundedFracCald} again shows $q_1=q_2$ in $W$. Hence, if one wants to establish uniqueness also for potentials $q_1,q_2\in M_{\gamma_0/\delta_0,+}(H^s\to H^{-s})$ satisfying $q_1\not\equiv q_2$ on $W$ and $W\cap \text{int}(\{q_1=q_2\})=\emptyset$, one would need to come up with a proof which does not rely on the separate determination of the coefficients as the one given in this article. 
    



    \section{Preliminaries}
    \label{sec: Preliminaries}

    Throughout this article $\Omega\subset \R^n$ is always an open set and the space dimension $n$ is fixed but otherwise arbitrary.
    
    \subsection{Fractional Laplacian and fractional conductivity operator}
    We define for $s> 0$ the fractional Laplacian of order $s$ by
    \begin{equation}
        (-\Delta)^su\vcentcolon = \ifourier(|\xi|^{2s}\widehat{u}),
    \end{equation}
    whenever the right hand side is well-defined. Here, $\mathcal{F}$ and $\mathcal{F}^{-1}$ denote the Fourier transform and the inverse Fourier transform, respectively. In this article we use the following convention
    \[
      \fourier u(\xi)\vcentcolon = \hat u(\xi) \vcentcolon =     \int_{\R^n} u(x)e^{-ix \cdot \xi} \,dx.
    \]
    If $u\colon\R^n\to\R$ is sufficiently regular and $s\in(0,1)$, the fractional Laplacian can be calculated via
    \begin{equation}
    \label{eq: singular int def frac Lap}
    \begin{split}
     (-\Delta)^su(x)&=C_{n,s}\,\text{p.v.}\int_{\R^n}\frac{u(x)-u(y)}{|x-y|^{n+2s}}\,dy\\
       &= -\frac{C_{n,s}}{2}\int_{\R^n}\frac{u(x+y)+u(x-y)-2u(x)}{|y|^{n+2s}}\,dy,
    \end{split}
    \end{equation}
    where $C_{n,s}>0$ is a normalization constant. Based on formula \eqref{eq: singular int def frac Lap}, we introduce the fractional conductivity operator $L_{\gamma}^s$ by
    \begin{equation}
    \label{eq: frac cond op}
        L_{\gamma}^su(x)=C_{n,s}\gamma^{1/2}(x)\,\text{p.v.}\int_{\R^n}\gamma^{1/2}(y)\frac{u(x)-u(y)}{|x-y|^{n+2s}}\,dy
    \end{equation}
    where $\gamma\colon\R^n\to\R_+$ is the so-called conductivity.

    \subsection{Sobolev spaces}
    The classical Sobolev spaces of order $k\in\N$ and integrability exponent $p\in [1,\infty]$ are denoted by $W^{k,p}(\Omega)$. Moreover, we let $W^{s,p}(\Omega)$ stand for the fractional Sobolev spaces, when $s\in \R_+\setminus\N$ and $1\leq p < \infty$. These spaces are also called Slobodeckij spaces or Gagliardo spaces. If $1\leq p<\infty$ and $s=k+\sigma$ with $k\in \N_0$, $0<\sigma<1$, then they are defined by
\[
    W^{s,p}(\Omega)\vcentcolon =\{\,u\in W^{k,p}(\Omega)\,;\, [\partial^{\alpha} u]_{W^{\sigma,p}(\Omega)}<\infty\quad \forall |\alpha|=k\, \},
\]
where 
\[
    [u]_{W^{\sigma,p}(\Omega)}\vcentcolon =\left(\int_{\Omega}\int_{\Omega}\frac{|u(x)-u(y)|^p}{|x-y|^{n+\sigma p}}\,dxdy\right)^{1/p}
\]
is the so-called Gagliardo seminorm. The Slobodeckij spaces are naturally endowed with the norm
\[
\|u\|_{W^{s,p}(\Omega)}\vcentcolon =\left(\|u\|_{W^{k,p}(\Omega)}^p+\sum_{|\alpha|=k}[\partial^{\alpha}u]_{W^{\sigma,p}(\Omega)}^p\right)^{1/p}.
\]

    We define the Bessel potential space $H^{s,p}(\R^n)$ for $1\leq p<\infty$, $s\in\R$ by
\begin{equation}
\label{eq: Bessel pot spaces}
    H^{s,p}(\R^n) \vcentcolon = \{ u \in \tempered(\R^n)\,;\, \vev{D}^su \in L^p(\R^n)\}
\end{equation}
 which we endow with the norm $\norm{u}_{H^{s,p}(\R^n)} \vcentcolon = \norm{\vev{D}^su}_{L^p(\R^n)}$. Here $\tempered(\R^n)$ denotes the space of tempered distributions, which is the dual of the space of Schwartz functions $\schwartz(\R^n)$, and $\langle D\rangle^s$ is the Fourier multiplier with symbol $\langle\xi\rangle^s=(1+|\xi|^2)^{s/2}$. In the special case $p=2$ and $0<s<1$, the spaces $H^{s,2}(\R^n)$ and $W^{s,2}(\R^n)$ coincide and they are commonly denoted by $H^s(\R^n)$.
 
 More concretely, the Gagliardo seminorm $[\,\cdot\,]_{H^s(\R^n)}$ and $\|\cdot\|_{\dot{H}^s(\R^n)}$ are equivalent on $H^s(\R^n)$ (cf.~\cite[Proposition~3.4]{DINEPV-hitchhiker-sobolev}).
 Throughout, this article we will assume that $0<s<\min(1,n/2)$ such that $H^s(\R^n)\hookrightarrow L^{2^*}(\R^n)$, where $2^*$ is the critical Sobolev exponent given by $2^*=\frac{2n}{n-2s}$. 
 
 If $\Omega\subset \R^n$, $F\subset\R^n$ are given open and closed sets, then we define the following local Bessel potential spaces:
\begin{equation}\label{eq: local bessel pot spaces}
\begin{split}
    \widetilde{H}^{s,p}(\Omega) &\vcentcolon = \mbox{closure of } C_c^\infty(\Omega) \mbox{ in } H^{s,p}(\R^n),\\
\end{split}
\end{equation}

We close this section by introducing the notion of domains bounded in one direction and recalling the related fractional Poincar\'e inequalities. We say that an open set $\Omega_\infty \subset\R^n$ of the form $\Omega_\infty=\R^{n-k}\times \omega$, where $n\geq k\geq 1$ and $\omega \subset \R^k$ is a bounded open set, is a \emph{cylindrical domain}. An open set $\Omega \subset \R^n$ is called \emph{bounded in one direction} if there exists a cylindrical domain $\Omega_\infty \subset \R^n$ and a rigid Euclidean motion $A(x) = Lx + x_0$, where $L$ is a linear isometry and $x_0 \in \R^n$, such that $\Omega \subset A\Omega_\infty$. Fractional Poincaré inequalities in Bessel potential spaces on domains bounded in one direction were recently studied in \cite{RZ2022unboundedFracCald}. In this article a $L^p$ generalization of the following result is established:

\begin{theorem}[{Poincar\'e inequality, \cite[Theorem~2.2]{RZ2022unboundedFracCald}}]
\label{thm: Poinc Unbounded Doms} Let $\Omega\subset\R^n$ be an open set that is bounded in one direction and $0<s<1$. Then there exists $C(n,s,\Omega)>0$ such that
    \begin{equation}
    \label{eq: poincare on L1}
        \|u\|^2_{L^2(\R^n)}\leq C[u]_{H^s(\R^n)}^2
    \end{equation}
    for all $u\in \widetilde{H}^{s}(\Omega)$.
\end{theorem}

\begin{remark}
    Let us note, that actually in \cite[Theorem~2.2]{RZ2022unboundedFracCald} the right hand side \eqref{eq: poincare on L1} is replaced by the seminorm $\|u\|_{\dot{H}^s(\R^n)}=\|(-\Delta)^{s/2}u\|_{L^{2}(\R^n)}$, but as already noted for $H^s(\R^n)$ functions these two expressions are equivalent.
\end{remark}

\subsection{Sobolev multiplier}
	
In this section we briefly introduce the Sobolev multipliers between the energy spaces $H^s(\R^n)$ and for more details we point to the book \cite{MS-theory-of-sobolev-multipliers} of Maz'ya and Shaposhnikova. 

Let $s,t\in\R$. If $f\in \distr(\R^n)$ is a distribution, we say that $f\in M(H^s\rightarrow H^t)$ whenever the norm 
\begin{equation}
    \|f\|_{s,t} \vcentcolon = \sup \{\abs{\ip{f}{u v}} \,;\, u,v \in C_c^\infty(\mathbb R^n), \norm{u}_{H^s(\R^n)} = \norm{v}_{H^{-t}(\R^n)} =1 \}
\end{equation}
is finite.
In the special case $t=-s$, we write $\|\cdot\|_s$ instead of $\|\cdot\|_{s,-s}$. Note that for any $f\in M(H^s\rightarrow H^t)$ and $u,v \in C_c^\infty(\mathbb R^n)$, we have the multiplier estimate
\begin{equation}
\label{multiplier-estimate}
    \abs{\ip{f}{uv}} \leq \|f\|_{s,t}\norm{u}_{H^s(\R^n)} \norm{v}_{H^{-t}(\R^n)}.
\end{equation}

By a density argument one easily sees that there is a unique linear multiplication map $m_f \colon H^s(\R^n) \to H^t(\R^n)$, $u\mapsto m_f(u)$. To simplify the notation we will write $fu$ instead of $m_f(u)$. 

Finally, we define certain subclasses of Sobolev multipliers from $H^s(\R^n)$ to $H^{-s}(\R^n)$. 
    For all $\delta > 0$ and $0<s<1$, we define the following convex sets
    \begin{equation}
    \label{eq: special classes of Sobolev multipliers}
    \begin{split}
        M_{\delta}(H^s \to H^{-s}) &\vcentcolon = \{\,q \in M(H^s \to H^{-s}) \,;\,\|q\|_{s} < \delta\,\},\\
        M_{+}(H^s \to H^{-s})& \vcentcolon =M(H^s\to H^{-s})\cap \distr_+(\R^n),\\
        M_{\delta,+}(H^s\to H^{-s})&\vcentcolon =M_{\delta}(H^s\to H^{-s})+M_{+}(H^s\to H^{-s}),
    \end{split}
    \end{equation}
    where $\distr_+(\R^n)$ denotes the non-negative distributions.
    
    Note that by definition of the multiplication map $u\mapsto fu$ one has $\langle qu,u\rangle \geq 0$ for all $u\in H^s(\R^n)$, whenever $q\in M_{+}(H^s \to H^{-s})$.

    \section{Well-posedness and DN map of forward problem}
    \label{subsec: forward}
    
    We start in Section~\ref{sec: basics} by recalling basic properties of the operator $L_{\gamma}^s$, like the fractional Liouville reduction, and then in Section~\ref{subsec: well-posedness and DN maps} we establish well-posedness results for the nonlocal optical tomography equation and the related fractional Schr\"odinger equation as well as introduce the associated DN maps.
    
    \subsection{Basics on the fractional conductivity operator \texorpdfstring{$L_{\gamma}^s$}{TEXT}}
    \label{sec: basics}  
    In this section, we recall several results related to the operator $L_{\gamma}^s$.
    
    First, for any uniformly elliptic coefficient $\gamma\in L^{\infty}(\R^n)$ and $0<s<1$, the operator $L_{\gamma}^s$ is weakly defined via the bilinear map $B_{\gamma}\colon H^s(\R^n)\times H^s(\R^n)\to\R$ with
    \begin{equation}
    \label{eq: bilinear frac cond eq}
        B_{\gamma}(u, v) 
        \vcentcolon =\frac{C_{n,s}}{2}\int_{\R^{2n}}\gamma^{1/2}(x)\gamma^{1/2}(y)\frac{(u(x)-u(y))(v(x)-v(y))}{|x-y|^{n+2s}}\,dxdy
    \end{equation}
    for all $u,v\in H^s(\R^n)$. Similarly, if $q\in M(H^s\to H^{-s})$, the bilinear map $B_q\colon H^s(\R^n)\times H^s(\R^n)\to \R$ representing the weak form of the fractional Schr\"odinger operator $(-\Delta)^s+q$ is defined via
    \begin{equation}
    \label{eq: schroedinger operator}
        B_{q}(u,v)\vcentcolon = \langle (-\Delta)^{s/2}u,(-\Delta)^{s/2}v\rangle_{L^2(\R^n)}+\langle qu,v\rangle
    \end{equation}
    for all $u,v\in H^s(\R^n)$.
    In \cite[Section~3]{RZ2022LowReg}, we showed that if the background deviation $m_{\gamma}=\gamma^{1/2}-1$ belongs to $H^{s,n/s}(\R^n)$, then the \emph{fractional Liouville reduction} is still valid, which was first established in \cite{RZ2022unboundedFracCald} for conductivities having background deviation in $H^{2s,n/2s}(\R^n)$ and hence $(-\Delta)^sm_{\gamma}\in L^{n/2s}(\R^n)$. More precisely, we established the following results:
    
    \begin{lemma}[{Fractional Liouville reduction}]
    \label{Lemma: fractional Liouville reduction}
        Let $0<s<\min(1,n/2)$, suppose $\Omega\subset\R^n$ is an open set and assume that the background deviation $m_{\gamma}=\gamma^{1/2}-1$ of the uniformly elliptic conductivity $\gamma\in L^{\infty}(\R^n)$ belongs to $H^{s,n/s}(\R^n)$. Then the following assertions hold:
        \begin{enumerate}[(i)]
        \item\label{estimate mathcal M} If $\mathcal{M}=m_{\gamma}$ or $\frac{m_{\gamma}}{m_{\gamma}+1}$, then $\mathcal{M}\in L^{\infty}(\R^n)\cap H^{s,n/s}(\R^n)$ and one has the estimate
        \begin{equation}
        \label{eq: continuity estimate background}
            \begin{split}
                \|\mathcal{M}v\|_{H^s(\R^n)}\leq C(\|\mathcal{M}\|_{L^{\infty}(\R^n)}+\|\mathcal{M}\|_{H^{s,n/s}(\R^n)})\|v\|_{H^s(\R^n)} 
            \end{split}
        \end{equation}
        for all $v\in H^s(\R^n)$ and some $C>0$. Moreover, if $u\in \widetilde{H}^s(\Omega)$, then there holds $\gamma^{\pm 1/2}u\in \widetilde{H}^s(\Omega)$
        \item\label{potential frac Liouville reduction} The distribution $q_{\gamma}=-\frac{(-\Delta)^sm_{\gamma}}{\gamma^{1/2}}$, defined by
    \[
        \langle q_{\gamma},\varphi\rangle\vcentcolon =-\langle (-\Delta)^{s/2}m_{\gamma},(-\Delta)^{s/2}(\gamma^{-1/2}\varphi)\rangle_{L^2(\R^n)}
    \]
    for all $\varphi\in C_c^{\infty}(\R^n)$, belongs to $M(H^s\rightarrow H^{-s})$. Moreover, for all $u,\varphi \in H^s(\R^n)$, we have
    \begin{equation}
        \langle q_{\gamma}u,\varphi\rangle = -\langle (-\Delta)^{s/2}m_{\gamma},(-\Delta)^{s/2}(\gamma^{-1/2}u\varphi)\rangle_{L^2(\R^n)}
    \end{equation}
    satisfying the estimate
    \begin{equation}
    \label{eq: bilinear estimate}
    \begin{split}
         |\langle q_{\gamma}u,\varphi\rangle|&\leq C(1+\|m_{\gamma}\|_{L^{\infty}(\R^n)}+\|m_{\gamma}\|_{H^{s,n/s}(\R^n)})
          \\ &\quad\quad \cdot \|m_{\gamma}\|_{H^{s,n/s}(\R^n)}\|u\|_{H^s(\R^n)}\|\varphi\|_{H^s(\R^n)}.
    \end{split}
    \end{equation}
\item\label{liouville red identity} There holds $B_{\gamma}(u,\varphi)=B_{q_{\gamma}}(\gamma^{1/2}u,\gamma^{1/2}\varphi)$ for all $u,\varphi\in H^s(\R^n)$, where $B_{q_{\gamma}}\colon H^s(\R^n)\times H^s(\R^n)\to \R$ is defined via \eqref{eq: schroedinger operator}.
\end{enumerate}
    \end{lemma}
    
    \subsection{Well-posedness results and DN maps}
    \label{subsec: well-posedness and DN maps}
    
    First, let us introduce for a given uniformly elliptic function $\gamma\in L^{\infty}(\R^n)$ and a potential $q\in M(H^s\to H^{-s})$ the bilinear map $B_{\gamma,q}\colon H^s(\R^n)\times H^s(\R^n)\to \R$ representing the weak form of the nonlocal optical tomography operator $L_{\gamma}^s+q$ via
    \begin{equation}
        B_{\gamma,q}(u,v)=B_{\gamma}(u,v)+\langle qu,v\rangle
    \end{equation}
    for all $u,v\in H^s(\R^n)$. As usual we say that a function $u\in H^s(\R^n)$ solves the Dirichlet problem
    \begin{equation}
    \begin{split}
        L_{\gamma}^s u+qu&= F\quad\text{in}\quad\Omega,\\
            u&= f\quad\text{in}\quad\Omega_e
    \end{split}
    \end{equation}
    for a given function $f\in H^s(\R^n)$ and $F\in (\widetilde{H}^s(\Omega))^*$ if there holds
    \[
        B_{\gamma,q}(u,\varphi)=\langle F,\varphi\rangle \enspace\text{for all}\enspace \varphi\in \widetilde{H}^s(\Omega)
    \]
    and $u-f\in \widetilde{H}^s(\Omega)$. We have the following well-posedness result for the nonlocal optical tomography equation.

    \begin{theorem}[Well-posedness and DN map for nonlocal optical tomography equation]
    \label{thm: well-posedness opt tom eq}
        Let $\Omega\subset \R^n$ be an open set which is bounded in one direction and $0<s<1$. Moreover, assume that the uniformly elliptic diffusion $\gamma\in L^{\infty}(\R^n)$ is bounded from below by $\gamma_0>0$ and the potential $q$ belongs to $M_{\gamma_0/\delta_0,+}(H^s\to H^{-s})$.
        Then the following assertions hold:
    \begin{enumerate}[(i)]
        \item\label{item 1 well-posedness opt tom eq} For all $f\in X\vcentcolon = H^s(\R^n)/\widetilde{H}^s(\Omega)$ there is a unique weak solution $u_f\in H^s(\R^n)$ of the fractional conductivity equation
        \begin{equation}
        \label{eq: nonlocal opt tomography problem}
        \begin{split}
            L_{\gamma}^su+qu&= 0\quad\text{in}\quad\Omega,\\
            u&= f\quad\text{in}\quad\Omega_e.
        \end{split}
        \end{equation}
        \item\label{item 2 DN map opt tom eq} The exterior DN map $\Lambda_{\gamma,q}\colon X\to X^*$ given by 
        \begin{equation}
        \label{eq: DN map opt tom eq}
        \begin{split}
            \langle \Lambda_{\gamma,q}f,g\rangle \vcentcolon =B_{\gamma,q}(u_f,g),
        \end{split}
        \end{equation}
        where $u_f\in H^s(\R^n)$ is the unique solution to  \eqref{eq: nonlocal opt tomography problem} with exterior value $f$, is a well-defined bounded linear map. 
    \end{enumerate}
    \end{theorem}
    
    \begin{remark}  
        In the above theorem and everywhere else in this article, we write $f$ instead of $[f]$ for elements of the trace space $X$. Let us note that on the right hand side of the formula \eqref{eq: DN map opt tom eq}, the function $g$ can be any representative of its equivalence class $[g]$. 
    \end{remark}
    
    \begin{proof}
        \ref{item 1 well-posedness opt tom eq}: First, let us note that the bilinear form $B_{\gamma}$ is continuous on $H^s(\R^n)$ and that any Sobolev multiplier $q\in M(H^s\to H^{-s})$ by the multiplier estimate \eqref{multiplier-estimate} induces a continuous bilinear form on $H^s(\R^n)$. Hence, $B_{\gamma,q}\colon H^s(\R^n)\times H^s(\R^n)\to \R$ is continuous. Moreover, as $q\in M_{\gamma_0/\delta,+}(H^s\to H^{-s})$ we may decompose $q$ as $q=q_1+q_2$, where $q_1\in M_{\gamma_0/\delta_0}(H^s\to H^{-s})$ and $q_2\in M_+(H^s\to H^{-s})$. Therefore, we can calculate
        \begin{equation}
        \label{eq: coercivity estimate potential q}
            \begin{split}
                B_{\gamma,q}(u,u)&\geq \gamma_0[u]_{H^s(\R^n)}^2+\langle q_1u,u\rangle+\langle q_2u,u\rangle\\
                &\geq \frac{\gamma_0}{2}\left([u]_{H^s(\R^n)}^2+C_{opt}^{-1}\|u\|_{L^2(\R^n)}^2\right)-|\langle q_1u,u\rangle|\\
                &\geq \frac{\gamma_0}{2\max(1,C_{opt})}\|u\|_{H^s(\R^n)}^2-\|q_1\|_{s}\|u\|_{H^s(\R^n)}^2\\
                &\geq (\gamma_0/\delta_0-\|q_1\|_{s})\|u\|_{H^s(\R^n)}^2=\alpha\|u\|_{H^s(\R^n)}^2
            \end{split}
        \end{equation}
        for any $u\in \widetilde{H}^s(\Omega)$, where we used the (optimal) fractional Poincar\'e inequality (see~Theorem~\ref{thm: Poinc Unbounded Doms} and eq.~\eqref{eq: optimal fractional Poincare constant}). Using the fact that $q_1\in M_{\gamma_0/\delta_0}(H^s\to H^{-s})$, we deduce $\alpha>0$ and hence the bilinear form $B_{\gamma,q}$ is coercive over $\widetilde{H}^s(\Omega)$.

        Next note that for given $f\in H^s(\R^n)$, the function $u\in H^s(\R^n)$ solves \eqref{eq: nonlocal opt tomography problem} if and only if $v=u-f\in H^s(\R^n)$ solves
        \begin{equation}
        \label{eq: hom nonlocal opt tomography problem}
        \begin{split}
            L_{\gamma}^sv+qv&= F\quad\text{in}\quad\Omega,\\
            v&= 0\quad\text{in}\quad\Omega_e
        \end{split}
        \end{equation}
        with $F=-(L_{\gamma}^sf+qf)\in (\widetilde{H}^s(\Omega))^*$. Now since $B_{\gamma,q}$ is a continuous, coercive bilinear form the Lax--Milgram theorem implies that \eqref{eq: hom nonlocal opt tomography problem} has a unique solution $v\in \widetilde{H}^s(\Omega)$ and so the same holds for \eqref{eq: nonlocal opt tomography problem}. Next, we show that if $f_1,f_2\in H^s(\R^n)$ satisfy $f_1-f_2\in\widetilde{H}^s(\Omega)$ then $u_{f_1}=u_{f_2}$ in $\R^n$, where $u_{f_j}\in H^s(\R^n)$, $j=1,2$, is the unique solution to \eqref{eq: nonlocal opt tomography problem} with exterior value $f_j$. Define $v=u_{f_1}-u_{f_2}\in \widetilde{H}^s(\Omega)$. Then $v$ solves 
        \begin{equation}
        \label{eq: uniqueness trace space}
        \begin{split}
            L_{\gamma}^sv+qv&= 0\quad\text{in}\quad\Omega,\\
            v&= 0\quad\text{in}\quad\Omega_e.
        \end{split}
        \end{equation}
        By testing \eqref{eq: uniqueness trace space} with $v$ and using the coercivity of $B_{\gamma,q}$ over $\widetilde{H}^s(\Omega)$, it follows that $v=0$ in $\R^n$. Hence, for any $f\in X$, there is a unique solution $u_f\in H^s(\R^n)$.
        
        \noindent \ref{item 2 DN map opt tom eq}: For any $f\in X$, let us define $\Lambda_{\gamma,q}f$ via the formula \eqref{eq: DN map opt tom eq}, where $g\in H^s(\R^n)$ is any representative of the related equivalence class in $X$. First, we verify that this map is well-defined. If $h\in H^s(\R^n)$ is any other representative, that is $g-h\in \widetilde{H}^s(\Omega)$, then since $u_f$ solves \eqref{eq: nonlocal opt tomography problem} we have
        \[
            B_{\gamma,q}(u_f,g)=B_{\gamma,q}(u_f,g-h)+B_{\gamma,q}(u_f,h)=B_{\gamma,q}(u_f,h)
        \]
        and so the expression for $\langle \Lambda_{\gamma,q}f,g\rangle$ is unambiguous. By the continuity of the bilinear form $B_{\gamma,q}$ it is easily seen that $\Lambda_{\gamma,q}f\in X^*$ for any $f\in X$.
    \end{proof}
    
    \begin{theorem}[Well-posedness and DN map for fractional Schr\"odinger equation]
    \label{thm: well-posedness for fractional Schrödinger type equation}
        Let $\Omega\subset \R^n$ be an open set which is bounded in one direction and $0<s<\min(1,n/2)$. Moreover, assume that the uniformly elliptic diffusion $\gamma\in L^{\infty}(\R^n)$ with lower bound $\gamma_0>0$ satisfies $m_{\gamma}\in H^{s,n/s}(\R^n)$ and the potential $q$ belongs to $M_{\gamma_0/\delta_0,+}(H^s\to H^{-s})$. Then the following assertions hold:
    \begin{enumerate}[(i)]
        \item\label{item 1 well-posedness schrödinger} The distribution $Q_{\gamma,q}$ defined by 
        \begin{equation}
        \label{eq: reduced potential}
            Q_{\gamma,q}=-\frac{(-\Delta)^sm_{\gamma}}{\gamma^{1/2}}+\frac{q}{\gamma}
        \end{equation}
        belongs to $M(H^s\to H^{-s})$.
        \item\label{item 2 well-posedness schrödinger} If $u\in H^s(\R^n)$, $f\in X$ and $v\vcentcolon =\gamma^{1/2}u,\,g\vcentcolon =\gamma^{1/2}f$, then $v\in H^s(\R^n), g\in X$ and $u$ is a solution of \eqref{eq: nonlocal opt tomography problem}
        if and only if $v$ is a weak solution of the fractional Schr\"odinger equation
        \begin{equation}
        \label{eq: FSE well-posedness schrödinger}    
            \begin{split}
            ((-\Delta)^s+Q_{\gamma,q})v&=0\quad\text{in}\quad\Omega,\\
            v&=g\quad\text{in}\quad\Omega_e.
        \end{split}
        \end{equation}
        \item\label{item 3 well-posedness schrödinger} Conversely, if $v\in H^s(\R^n), g\in X$ and $u\vcentcolon =\gamma^{-1/2}v,\,f\vcentcolon =\gamma^{-1/2}g$, then $v$ is a weak solution of \eqref{eq: FSE well-posedness schrödinger}  if and only if $u$ is a weak solution of \eqref{eq: nonlocal opt tomography problem}.
        \item\label{item 4 well-posedness schrödinger} For all $f\in X$ there is a unique weak solution $v_g\in H^s(\R^n)$ of the fractional Schr\"odinger equation \eqref{eq: FSE well-posedness schrödinger}.
        \item\label{item 5 well-posedness schrödinger} The exterior DN map $\Lambda_{Q_{\gamma,q}}\colon X\to X^*$ given by 
        \begin{equation}
        \label{eq: well-defined DN map Schrödinger}
        \begin{split}
            \langle \Lambda_{Q_{\gamma,q}}f,g\rangle\vcentcolon =B_{Q_{\gamma,q}}(v_f,g),
        \end{split}
        \end{equation}
        where $v_f\in H^s(\R^n)$ is the unique solution to \eqref{eq: FSE well-posedness schrödinger} with exterior value $f$, is a well-defined bounded linear map. 
    \end{enumerate}
    \end{theorem}
    
    \begin{proof}
        \ref{item 1 well-posedness schrödinger}: Since $q\in M(H^s\to H^{-s})$, we can estimate 
        \begin{equation}
        \label{eq: continuity estimate potential}
        \begin{split}
            |\langle q/\gamma u,v\rangle|&=|\langle q (\gamma^{-1/2}u),\gamma^{-1/2}v\rangle|\\
            &\leq \|q\|_s\|\gamma^{-1/2}u\|_{H^s(\R^n)}\|\gamma^{-1/2}v\|_{H^s(\R^n)}\\
            &\leq \|q\|_s\|(1-\frac{m_{\gamma}}{m_{\gamma}+1})u\|_{H^s(\R^n)}\|(1-\frac{m_{\gamma}}{m_{\gamma}+1})v\|_{H^s(\R^n)}\\
            &\leq C\|q\|_s\|u\|_{H^s(\R^n)}\|v\|_{H^s(\R^n)}
        \end{split}
        \end{equation}
        for all $u,v\in H^s(\R^n)$, where we used that the assertion \ref{estimate mathcal M} of Lemma~\ref{Lemma: fractional Liouville reduction} implies $\gamma^{-1/2}w\in H^s(\R^n)$ for all $w\in H^s(\R^n)$ with $\|\frac{m_{\gamma}}{m_{\gamma}+1}w\|_{H^s(\R^n)}\leq C\|w\|_{H^s(\R^n)}$ for some constant $C>0$ only depending polynomially on the $L^{\infty}$ and $H^{s,n/s}$ norm of $\frac{m_{\gamma}}{m_{\gamma}+1}$. Now, the estimate \eqref{eq: continuity estimate potential} can be used to see that $q/\gamma$ is a distribution and belongs to $M(H^s\to H^{-s})$. On the other hand, by the statement \ref{potential frac Liouville reduction} of Lemma~\ref{Lemma: fractional Liouville reduction} we know that $q_{\gamma}=\frac{(-\Delta)^sm_{\gamma}}{\gamma^{1/2}}\in M(H^s\to H^{-s})$. This in turn implies $Q_{\gamma}\in M(H^s\to H^{-s})$.
        
        \noindent\ref{item 2 well-posedness schrödinger}: The assertions $v\in H^s(\R^n)$, $g\in X$ and $u-f\in \widetilde{H}^s(\Omega)$ if and only if $v-g\in \widetilde{H}^s(\Omega)$ are direct consequences of the property \ref{estimate mathcal M} of Lemma~\ref{Lemma: fractional Liouville reduction}. Furthermore, the fact that $u$ solves $L_{\gamma}^su+qu=0$ in $\Omega$ if and only if $v$ solves $(-\Delta)^sv+Q_{\gamma,q}=0$ in $\Omega$ follows by the definition of $Q_{\gamma,q}$, \ref{liouville red identity} and \ref{estimate mathcal M} of Lemma~\ref{Lemma: fractional Liouville reduction}.
        
        \noindent\ref{item 3 well-posedness schrödinger}: The proof of this fact is essentially the same as for \ref{item 2 well-posedness schrödinger} and therefore we drop it.
        
       \noindent\ref{item 4 well-posedness schrödinger}: By \ref{item 3 well-posedness schrödinger}, we know that $v\in H^s(\R^n)$ solves \eqref{eq: FSE well-posedness schrödinger}  if and only if $u$ solves \eqref{eq: nonlocal opt tomography problem} with exterior value $f=\gamma^{1/2}g$. The latter Dirichlet problem is well-posed by Theorem~\ref{thm: well-posedness opt tom eq} and hence it follows from \ref{item 2 well-posedness schrödinger} and \ref{item 2 well-posedness schrödinger} that the unique solution of \eqref{eq: FSE well-posedness schrödinger}  is given by $v_g=\gamma^{1/2}u_{\gamma^{-1/2}g}\in H^s(\R^n)$.
        
        \noindent\ref{item 5 well-posedness schrödinger}: The fact that $\Lambda_{Q_{\gamma,q}}$ defined via formula \eqref{eq: well-defined DN map Schrödinger} is well-defined follows from the properties \ref{item 4 well-posedness schrödinger}, \ref{item 1 well-posedness schrödinger} and the same calculation as in the proof of Theorem~\ref{thm: well-posedness opt tom eq}, \ref{item 2 DN map opt tom eq}.
    \end{proof}
    
    \begin{remark}
    \label{remark: interior source problem}
        Let us note that essentially the same proofs as in Theorem~\ref{thm: well-posedness opt tom eq} and \ref{thm: well-posedness for fractional Schrödinger type equation}, can be used to show that
        \[
        \begin{split}
            L_{\gamma}^su+qu&= F\quad\text{in}\quad\Omega,\\
            u&= u_0\quad\text{in}\quad\Omega_e
        \end{split}
        \]
        and
        \[
            \begin{split}
            ((-\Delta)^s+Q_{\gamma,q})v&=G\quad\text{in}\quad\Omega,\\
            v&=v_0\quad\text{in}\quad\Omega_e.
        \end{split}
        \]
        for all $u_0,v_0\in H^s(\R^n)$ and $F,G\in(\widetilde{H}^s(\Omega))^*$ are well-posed.
    \end{remark}

    \section{Inverse problem}
    \label{sec: inverse problem}
    
    In Section~\ref{sec: uniquness} we first prove Theorem~\ref{main thm} and hence providing an answer to Question~\ref{question uniqueness}. We establish this result in four steps. First, in Section~\ref{subsec: exterior reconstruction} we extend the exterior determination result of the fractional conductivity equation to the nonlocal tomography equation (Theorem~\ref{thm: exterior reconstruction}). Then in Lemma~\ref{lemma: relation of sols} we show that $\gamma_1^{1/2}u_f^{(1)}$ and $\gamma_2^{1/2}u_f^{(2)}$ coincide in $\R^n$ whenever $\gamma_1=\gamma_2$, $q_1=q_2$ in the measurement set and generate the same DN data. These two preparatory steps then allow us to prove that the diffusion coefficients are the same in $\R^n$ (Section~\ref{subsec: determination of diffusion coeff}) and to conclude that in that case also the absorption coefficients are necessarily identical (Section~\ref{subsubsec: equality of q}). Then in Section~\ref{sec: non-uniqueness}, we provide an answer to Question~\ref{question non-uniqueness}. Following a similar strategy as in \cite{RZ2022LowReg}, we first derive a characterization of the uniqueness in the inverse problem for the nonlocal optical tomography equation and then use this to construct counterexamples to uniqueness when the potentials are non-equal in the measurement set (see~Theorem~\ref{thm: non uniqueness}).
    
    \subsection{Uniqueness}
    \label{sec: uniquness}

    \subsubsection{Exterior reconstruction formula}
    \label{subsec: exterior reconstruction}
    
    The main result of this section is the following reconstruction formula in the exterior.
    
    \begin{theorem}[Exterior reconstruction formula]
    \label{thm: exterior reconstruction}
        Let $\Omega\subset \R^n$ be an open set which is bounded in one direction, $W\subset\Omega_e$ a measurement set and $0<s<\min(1,n/2)$. Assume that the uniformly elliptic diffusion $\gamma\in L^{\infty}(\R^n)$, which is bounded from below by $\gamma_0>0$, and the potential $q\in M_{\gamma_0/\delta_0,+}(H^s\to H^{-s})$ satisfy the following additional properties
        \begin{enumerate}[(i)]
            \item\label{prop 1 reconstruction} $\gamma$ is a.e. continuous in $W$
            \item\label{prop 2 reconstruction} and $q\in L^p_{loc}(W)$ for some $\frac{n}{2s}<p\leq \infty$.
        \end{enumerate}
        Then for a.e. $x_0\in W$ there exists a sequence $(\Phi_N)_{N\in\N}\subset C_c^{\infty}(W)$ such that
        \begin{equation}
        \label{eq: reconstruction formula}
            \gamma(x_0)=\lim_{N\to\infty}\langle \Lambda_{\gamma,q}\Phi_N,\Phi_N\rangle.
        \end{equation}
    \end{theorem}
    
    Before giving the proof of this result, we prove the following interpolation estimate:
    
    \begin{lemma}[Interpolation estimate for the potential term]\label{lem: cont potential term}
        Let $0 < s  < \min(1,n/2)$ and assume $W\subset\R^n$ is a non-empty open set. If $q\in M(H^s\to H^{-s})\cap L^p_{loc}(W)$ for some $\frac{n}{2s} < p \le \infty$, then for any $V\Subset W$ the following estimate holds
        \begin{equation}
        \label{eq: potential goes to zero estimate}
            |\langle qu,v\rangle|\leq C \|u\|^{1-\theta}_{H^s(\R^n)}\|u\|_{L^2(V)}^{\theta} \|v\|_{H^s(\R^n)}
        \end{equation}
        for all $u,v\in C_c^{\infty}(V)$ and some $C>0$, where $\theta\in (0,1]$ is given by 
        \begin{equation}
            \theta=\begin{cases}
                2-\frac{n}{sp},&\enspace\text{if}\enspace \frac{n}{2s}<p\leq \frac{n}{s}, \\
                1,&\enspace\text{otherwise}.
            \end{cases}
        \end{equation}
    \end{lemma}
    \begin{proof}
    Without loss of generality we can assume that there holds $\frac{n}{2s}<p\leq \frac{n}{s}$. First, by H\"older's inequality and Sobolev's embedding we have 
    \begin{equation}
    \label{eq: first estimate}
        |\langle qu,v\rangle|\leq \|qu\|_{L^{\frac{2n}{n+2s}}(V)}\|v\|_{L^{\frac{2n}{n-2s}}(V)}\leq C\|qu\|_{L^{\frac{2n}{n+2s}}(V)}\|v\|_{H^s(\R^n)}.
    \end{equation}
    Next, observe that if $\theta=2-\frac{n}{sp}\in (0,1]$, then there holds
    \[
        \frac{n+2s}{2n}=\frac{1}{p}+\frac{1-\theta}{\frac{2n}{n-2s}}+\frac{\theta}{2}.
    \]
    Therefore, by interpolation in $L^q$ and Sobolev's embedding we can estimate
    \begin{equation}
    \label{eq: second estimate}
    \begin{split}
        \|qu\|_{L^{\frac{2n}{n+2s}}(V)}&\leq \|q\|_{L^p(V)}\|u\|_{L^{\frac{2n}{n-2s}}(V)}^{1-\theta}\|u\|_{L^2(V)}^{\theta}\\
        &\leq C\|q\|_{L^p(V)}\|u\|_{H^s(\R^n)}^{1-\theta}\|u\|_{L^2(V)}^{\theta}.
    \end{split}
    \end{equation}
    Combining the estimates \eqref{eq: first estimate} and \eqref{eq: second estimate}, we obtain \eqref{eq: potential goes to zero estimate}.

    \end{proof}
    
    \begin{proof}[Proof of Theorem~\ref{thm: exterior reconstruction}]
        Let $x_0\in W$ be such that $\gamma$ is continuous at $x_0$. By \cite[Theorem~1.1]{KLZ22FracpLap}, there exists a sequence $(\Phi_N)_{N\in\N}\subset C_c^{\infty}(W)$ satisfying the following conditions:
        \begin{enumerate}[(i)]
            \item\label{support cond} $\supp(\Phi_N)\to \{x_0\}$ as $N\to\infty$,
            \item\label{normalization cond} $[\Phi_N]_{H^s(\R^n)}=1$ for all $N\in\N$
            \item\label{convergence cond} and $\Phi_N\to 0$ in $H^t(\R^n)$ as $N\to\infty$ for all $0\leq t<s$.
        \end{enumerate}
        The last condition implies that $\Phi_N\to 0$ in $L^p(\R^n)$ for all $1\leq p <\frac{2n}{n-2s}$ as $N\to\infty$. Next, let $u_N\in H^s(\R^n)$ be the unique solution to
        \begin{equation}
        \begin{split}
            L^s_{\gamma}u+qu&= 0\quad\enspace\text{in}\enspace\Omega,\\
            u&= \Phi_N\,\enspace\text{in}\enspace\Omega_e.
        \end{split}
        \end{equation}
        By linearity $v_N\vcentcolon = u_N-\Phi_N\in \widetilde{H}^s(\Omega)$ is the unique solution to
        \begin{equation}
        \label{eq: sol hom ext cond}
        \begin{split}
            L_{\gamma}^s v+qv&= -B_{\gamma,q}(\Phi_N,\cdot)\enspace\text{in}\enspace\Omega,\\
            v&= 0\quad\,\,\,\quad\quad\quad\quad\text{in}\enspace\Omega_e.
        \end{split}
        \end{equation}
        One easily sees that $B_{\gamma,q}(\Phi_N,\cdot)\in (\widetilde{H}^s(\Omega))^*$. Similarly as in \cite[Lemma~3.1]{RZ2022LowReg}, for any $v\in \widetilde{H}^s(\Omega)$ we may calculate
        \allowdisplaybreaks
        \[
        \begin{split}
            &|B_{\gamma,q}(\Phi_N,v)|=|B_{\gamma}(\Phi_N,v)|=C\left|\int_{W\times\Omega}\gamma^{1/2}(x)\gamma^{1/2}(y)\frac{\Phi_N(x)v(y)}{|x-y|^{n+2s}}\,dxdy\right|\\
            &\quad\leq C\int_{\Omega}\gamma^{1/2}(y)|v(y)|\left(\int_W\frac{\gamma^{1/2}(x)|\Phi_N(x)|}{|x-y|^{n+2s}}\,dx\right)dy\\
            &\quad\leq C\|\gamma\|_{L^{\infty}(\Omega\cup W)}\|v\|_{L^2(\Omega)}\left\|\int_W\frac{|\Phi_N(x)|}{|x-y|^{n+2s}}\,dx\right\|_{L^2(\Omega)}\\
            &\quad\leq C\|\gamma\|_{L^{\infty}(\Omega\cup W)}\|v\|_{L^2(\Omega)}\int_W|\Phi_N(x)|\left(\int_{\Omega}\frac{dy}{|x-y|^{n+2s}}\,dy\right)^{1/2}\,dx\\
            &\quad\leq C\|\gamma\|_{L^{\infty}(\Omega\cup W)}\|v\|_{L^2(\Omega)}\int_W|\Phi_N(x)|\left(\int_{(B_r(x))^c}\frac{dy}{|x-y|^{n+2s}}\,dy\right)^{1/2}\,dx\\
            &\quad\leq \frac{C}{r^{\frac{n+4s}{2}}}\|\gamma\|_{L^{\infty}(\R^n)}\|v\|_{L^2(\Omega)}\|\Phi_N\|_{L^1(W)}.
        \end{split}
        \]
        In the above estimates we used that $\gamma\in L^{\infty}(\R^n)$ is uniformly elliptic, $\supp(\Phi_N)\subset \supp(\Phi_1)\Subset W$ (see~\ref{support cond}), H\"older's and Minkowski's inequality and set $r\vcentcolon = \dist(\Omega,\supp(\Phi_1))>0$. This implies
        \begin{equation}
        \label{eq: bounded on norm}
            \|B_{\gamma,q}(\Phi_N,\cdot)\|_{(\widetilde{H}^s(\Omega))^*}\leq \frac{C}{r^{\frac{n+4s}{2}}}\|\gamma\|_{L^{\infty}(\R^n)}\|\Phi_N\|_{L^1(W)}.
        \end{equation}
        Now, testing equation \eqref{eq: sol hom ext cond} by $v_N\in\widetilde{H}^s(\Omega)$, using the fractional Poincar\'e inequality (see~Theorem~\ref{thm: Poinc Unbounded Doms}), the uniform ellipticity of $\gamma$ and the coercivity estimate \eqref{eq: coercivity estimate potential q}, we get
        \[
        \begin{split}
            \|v_N\|_{H^s(\R^n)}^2&\leq C|B_{\gamma,q}(\Phi_N,v_N)|\leq C\|B_{\gamma,q}(\Phi_N,\cdot)\|_{(\widetilde{H}^s(\Omega))^*}\|v_N\|_{H^s(\R^n)}\\
            &\leq \frac{C}{r^{\frac{n+4s}{2}}}\|\gamma\|_{L^{\infty}(\R^n)}\|\Phi_N\|_{L^1(W)}\|v_N\|_{H^s(\R^n)},
        \end{split}
        \]
        which in turn implies
        \[
            \|v_N\|_{H^s(\R^n)}\leq \frac{C}{r^{\frac{n+4s}{2}}}\|\gamma\|_{L^{\infty}(\R^n)}\|\Phi_N\|_{L^1(W)}.
        \]
        Recalling that $v_N=u_N-\Phi_N$ and the property \ref{convergence cond} of the sequence $\Phi_N\in C_c^{\infty}(W)$, we deduce
        \begin{equation}
        \label{eq: conv to zero of diff}
            \|u_N-\Phi_N\|_{H^s(\R^N)}\to 0\quad\text{as}\quad N\to\infty.
        \end{equation}
        Let us next define the energy
        \[
            E_{\gamma,q}(v)\vcentcolon = B_{\gamma,q}(v,v)
        \]
        for any $v\in H^s(\R^n)$. Using the computation in the proof of \cite[Theorem~3.2]{RZ2022LowReg} we have
        \begin{equation}
        \label{eq: concentration of energy}
            \begin{split}
                \lim_{N\to\infty}E_{\gamma,q}(\Phi_N)&=\lim_{N\to\infty}B_{\gamma}(\Phi_N,\Phi_N)+\lim_{N\to\infty}\langle q\Phi_N,\Phi_N\rangle_{L^2(\R^n)}\\
                &=\lim_{N\to\infty}B_{\gamma}(\Phi_N,\Phi_N)=\gamma(x_0)
            \end{split}
        \end{equation}
        where we used Lemma~\ref{lem: cont potential term} and the properties \ref{normalization cond}, \ref{convergence cond} of the sequence $(\Phi_N)_{N\in\N}$ to see that the term involving the potential $q$ vanishes.
        On the other hand, we can rewrite the DN map as follows
        \[
        \begin{split}
            \langle \Lambda_{\gamma,q}\Phi_N,\Phi_N\rangle &=B_{\gamma,q}(u_N,\Phi_N)=B_{\gamma,q}(u_N,u_N)\\
            &=E_{\gamma,q}(u_N-\Phi_N)+2B_{\gamma,q}(u_N-\Phi_N,\Phi_N)+E_{\gamma,q}(\Phi_N).
        \end{split}
        \]
        Thus, arguing as above for the convergence $E_{\gamma,q}(\Phi_N)\to\gamma(x_0)$, we see that the first two terms on the right hand side vanish in the limit $N\to\infty$ and we can conclude the proof.
    \end{proof}

    \subsubsection{Uniqueness of the diffusion coefficient $\gamma$}
    \label{subsec: determination of diffusion coeff}
    
    \begin{lemma}[Relation of solutions]
    \label{lemma: relation of sols}
        Let $\Omega\subset \R^n$ be an open set which is bounded in one direction, suppose $W_1,W_2\subset\Omega_e$ are two measurement sets and $0<s<\min(1,n/2)$. Assume that the uniformly elliptic diffusions $\gamma,\gamma_2\in L^{\infty}(\R^n)$ with lower bound $\gamma_0>0$ satisfy $m_{\gamma_1},m_{\gamma_2}\in H^{s,n/s}(\R^n)$ and the potentials $q_1,q_2$ belong to $M_{\gamma_0/\delta_0,+}(H^s\to H^{-s})$. If $\gamma_1|_{W_2} = \gamma_2|_{W_2}$ and $\Lambda_{\gamma_1, q_1} f|_{W_2} = \Lambda_{\gamma_2, q_2} f|_{W_2}$ for some $f\in \widetilde{H}^s(W_1)$ with $W_2\setminus \supp(f)\neq \emptyset$, then there holds
        \begin{equation}
            \gamma_1^{1/2}u_f^{(1)} = \gamma_2^{1/2}u_f^{(2)}\enspace \text{a.e.\@ in }\R^n
        \end{equation}
        where, for $j = 1, 2$, $u_f^{(j)}\in H^s(\R^n)$ is the unique solution of \begin{equation}
        \label{eq: NOTE relation solutions}
        \begin{split}
            L_{\gamma_j}^su+q_ju&= 0\quad\text{in}\quad\Omega,\\
            u&= f\quad\text{in}\quad\Omega_e
        \end{split}
        \end{equation}
        (see~Theorem~\ref{thm: well-posedness opt tom eq}).
    \end{lemma}

    \begin{proof}
        First let $\gamma,q$ satisfy the assumptions of Lemma~\ref{lemma: relation of sols} and assume that $f,g\in H^s(\R^n)$ have disjoint support. Then there holds
        \begin{equation}
        \label{eq: disjoint support}
            \begin{split}
                B_{\gamma,q}(f,g)&=B_{\gamma}(f,g)=C_{n,s}\int_{\R^{2n}}\gamma^{1/2}(x)\gamma^{1/2}(y)\frac{f(x)g(y)}{|x-y|^{n+2s}}\,dxdy\\
                &=\langle (-\Delta)^{s/2}(\gamma^{1/2}f),(-\Delta)^{s/2}(\gamma^{1/2}g)\rangle_{L^2(\R^n)}.
            \end{split}
        \end{equation}
        Now, let $f\in \widetilde{H}^s(W_1)$ and $u_f^{(j)}\in H^s(\R^n)$ for $j=1,2$ be as in the statement of the lemma. Set $V\vcentcolon = W_2\setminus\supp(f)$ and take any $\varphi\in \widetilde{H}^s(V)$. Then we have $\supp(u_f^{(j)})\cap \supp(\varphi)=\emptyset$ and the assumption that the DN maps coincide implies
        \[
            \begin{split}
                B_{\gamma_1,q_1}(u_f^{(1)},\varphi)=\langle \Lambda_{\gamma_1,q_1}f,\varphi\rangle =\langle \Lambda_{\gamma_2,q_2}f,\varphi\rangle=B_{\gamma_2,q_2}(u_f^{(2)},\varphi).
            \end{split}
        \]
        By \eqref{eq: disjoint support} and the assumption $\gamma_1=\gamma_2$ on $W_2$, this is equivalent to
        \[
            \langle (-\Delta)^{s/2}(\gamma_1^{1/2}u_f^{(1)}-\gamma_2^{1/2}u_f^{(2)}),(-\Delta)^{s/2}(\gamma_1^{1/2}\varphi)\rangle_{L^2(\R^n)}=0
        \]
        for all $\varphi\in \widetilde{H}^s(V)$. By our assumptions on the diffusion coefficients $\gamma_j$ and Lemma~\ref{Lemma: fractional Liouville reduction}, we can replace $\varphi$ by $g=\gamma_1^{-1/2}\varphi$ to obtain
        \[
            \langle (-\Delta)^{s/2}(\gamma_1^{1/2}u_f^{(1)}-\gamma_2^{1/2}u_f^{(2)}),(-\Delta)^{s/2}g\rangle_{L^2(\R^n)}=0
        \]
        for all $g\in \widetilde{H}^s(V)$. We know that $\gamma_1^{1/2}u_f^{(1)}-\gamma_2^{1/2}u_f^{(2)}=0$ on $V$ as $u_f^{(j)}=0$ on $V$. Therefore, Lemma~\ref{Lemma: fractional Liouville reduction} and the usual UCP for the fractional Laplacian for $H^s$ functions implies $\gamma_1^{1/2}u_f^{(1)}=\gamma_2^{1/2}u_f^{(2)}$ a.e. in $\R^n$.
    \end{proof}
    
     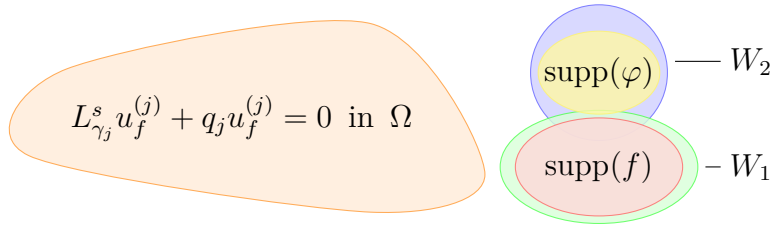
\begin{figure}[!ht]
    \centering
    \begin{tikzpicture}
    \filldraw[color=blue!50, fill=blue!15, xshift=11cm, yshift=1.35cm] (3,-2.5) ellipse (0.9 and 0.9);   
    \node[xshift=13cm, yshift=1.5cm] at (3,-2.5) {$\raisebox{-.35\baselineskip}{\ensuremath{W_2}}$};
    \filldraw[color=green!70, fill=green!15, xshift=11cm, yshift=0.1cm, opacity=0.8] (3,-2.5) ellipse (1.3 and 0.75);   
    \filldraw[color=red!60, fill=red!15, xshift=11cm, yshift=0.1cm, opacity=0.8] (3,-2.5) ellipse (1.1 and 0.65);  
    \node[xshift=11cm, yshift=0.1cm] at (3,-2.5) {$\raisebox{-.35\baselineskip}{\ensuremath{\supp(f)}}$};
    \node[xshift=13cm, yshift=0.1cm] at (3,-2.5) {$\raisebox{-.35\baselineskip}{\ensuremath{W_1}}$};
    \draw[xshift=13cm, yshift=0.1cm] (2.4,-2.5) -- (2.6,-2.5);
    \draw[xshift=13cm, yshift=0.1cm] (2,-1.1) -- (2.6,-1.1);
    \filldraw[color=yellow!70, fill=yellow!45, xshift=11cm, yshift=1.35cm, opacity=0.8] (3,-2.5) ellipse (0.8 and 0.55);  
    \node[xshift=11cm, yshift=1.35cm] at (3,-2.5) {$\raisebox{-.35\baselineskip}{\ensuremath{\supp(\varphi)}}$};
    \filldraw [color=orange!80, fill = orange!15, xshift=8cm, yshift=-2cm,opacity=0.8] plot [smooth cycle] coordinates {(-1,0.9) (3,1.5) (4.5,-0.5) (3,-1) (-1.5,-0.25)};
    \node[xshift=3cm] at (6.3,-1.75) {$\raisebox{-.35\baselineskip}{\ensuremath{L_{\gamma_j}^su_f^{(j)}+q_ju_f^{(j)}=0\enspace\text{in}\enspace\Omega}}$};
\end{tikzpicture}
\caption{\begin{small} A graphical illustration of the sets and functions used in the proof of Lemma~\ref{lemma: relation of sols}. \end{small}}
\end{figure}

    \begin{theorem}[Uniqueness of $\gamma$]
    \label{thm: uniqueness of gamma}
        Let $0 < s < \min(1,n/2)$, suppose $\Omega\subset \R^n$ is a domain bounded in one direction and let $W_1,W_2\subset \Omega$ be two non-disjoint measurement sets. Assume that the diffusions $\gamma_1, \gamma_2\in L^{\infty}(\R^n)$ with background deviations $m_{\gamma_1},m_{\gamma_2}\in H^{s,n/s}(\R^n)$ and potentials $q_1,q_2\in \distr(\R^n)$ satisfy
        \begin{enumerate}[(i)]
            \item\label{uniform ellipticity diffusions} $\gamma_1,\gamma_2$ are uniformly elliptic with lower bound $\gamma_0>0$,
            \item\label{continuity diffusions} $\gamma_1, \gamma_2$ are a.e. continuous in $W_1\cap W_2$,
            \item\label{integrability potentials} $q_1,q_2\in M_{\gamma_0/\delta_0,+}(H^s\to H^{-s})\cap L^p_{loc}(W_1\cap W_2)$ for $\frac{n}{2s}< p\leq \infty$
            \item\label{equal potentials in measurement sets} and $q_1|_{W_1\cap W_2}=q_2|_{W_1\cap W_2}$.
        \end{enumerate} 
        If $\Lambda_{\gamma_1,q_1}f|_{W_2}=\Lambda_{\gamma_2,q_2}f|_{W_2}$ for all $f\in C_c^{\infty}(W_1)$, then there holds $\gamma_1=\gamma_2$ in $\R^n$.
    \end{theorem}

    \begin{proof}
        Let $W\vcentcolon = W_1\cap W_2$. Then Theorem~\ref{thm: exterior reconstruction} ensures that $\gamma_1=\gamma_2$ on $W$. Next choose $V\Subset W$ and let $f\in \widetilde{H}^s(V)$. By assumption there holds 
        \[
        \begin{split}
            0 &=\langle (\Lambda_{\gamma_1,q_1}-\Lambda_{\gamma_2,q_2} )f,f\rangle =B_{\gamma_1,q_1}(u_f^{(1)},f)-B_{\gamma_2,q_2}(u_f^{(2)},f)\\
            &=B_{\gamma_1}(u_f^{(1)},f)-B_{\gamma_2}(u_f^{(2)},f)+\langle (q_1-q_2)f,f\rangle\\
            &=\langle (-\Delta)^{s/2}(\gamma_1^{1/2}u_f^{(1)}),(-\Delta)^{s/2}(\gamma_1^{1/2}f)\rangle_{L^2(\R^n)}\\
            &\quad-\left\langle \frac{(-\Delta)^sm_{\gamma_1}}{\gamma_1^{1/2}}\gamma_1^{1/2}u_f^{(1)},\gamma_1^{1/2}f\right\rangle\\
            &\quad +\langle (-\Delta)^{s/2}(\gamma_2^{1/2}u_f^{(2)}),(-\Delta)^{s/2}(\gamma_2^{1/2}f)\rangle_{L^2(\R^n)}\\
            &\quad-\left\langle \frac{(-\Delta)^sm_{\gamma_2}}{\gamma_2^{1/2}}\gamma_2^{1/2}u_f^{(2)},\gamma_2^{1/2}f\right\rangle\\
            &\quad +\langle (q_1-q_2)f,f\rangle\\
            &=\left\langle \frac{(-\Delta)^s(m_{\gamma_2}-m_{\gamma_1})}{\gamma_1^{1/2}}\gamma_1^{1/2}f,\gamma_1^{1/2}f\right\rangle+\langle (q_1-q_2)f,f\rangle\\
            &\quad +\langle (-\Delta)^{s/2}(\gamma_1^{1/2}u_f^{(1)}-\gamma_2^{1/2}u_f^{(2)}),(-\Delta)^{s/2}(\gamma_1^{1/2}f)\rangle_{L^2(\R^n)},
        \end{split}
        \]
        where in the fourth equality sign we used the fractional Liouville reduction (Lemma~\ref{Lemma: fractional Liouville reduction}, \ref{liouville red identity}) and in the fifth equality sign $\gamma_1=\gamma_2$ in $W$. By Lemma~\ref{lemma: relation of sols} with $W_1=V$ and $W_2=W\setminus\overline{V}$, the term in the last line vanishes. Moreover, since $q_1=q_2$ in $W$, the term involving the potentials is zero as well. Using the polarization identity, we deduce that there holds
        \[
            \left\langle \frac{(-\Delta)^s(m_{\gamma_2}-m_{\gamma_1})}{\gamma_1^{1/2}}\gamma_1^{1/2}f,\gamma_1^{1/2}g\right\rangle=0
        \]
        for all $f,g\in \widetilde{H}^s(V)$.  In particular, by first changing $f\mapsto \gamma_1^{-1/2}f\in \widetilde{H}^s(V)$ and $g\mapsto \gamma_1^{-1/2}g\in \widetilde{H}^s(V)$ (see~Lemma~\ref{Lemma: fractional Liouville reduction}, \ref{estimate mathcal M}) and then selecting $U\Subset V$, $g\in C_c^{\infty}(V)$ with $0\leq g\leq 1$, $g|_{\overline{U}}=1$, this implies 
        \[
            \left\langle (-\Delta)^s(m_{\gamma_2}-m_{\gamma_1}),\gamma_1^{-1/2}f\right\rangle=0
        \]
        for all $f\in \widetilde{H}^s(U)$. Using again the assertion \ref{estimate mathcal M} of Lemma~\ref{Lemma: fractional Liouville reduction}, we deduce
        \[
            \left\langle (-\Delta)^s(m_{\gamma_2}-m_{\gamma_1}),f\right\rangle=0
        \]
        for all $f\in \widetilde{H}^s(U)$. Hence, $m=m_{\gamma_2}-m_{\gamma_1}\in H^{s,n/s}(\R^n)$ satisfies
        \[
            (-\Delta)^sm=m=0\quad \text{in}\quad U.
        \]
        Now, the UCP for the fractional Laplacian in $H^{s,n/s}(\R^n)$ (see~\cite[Theorem~2.2]{KRZ2022Biharm}) guarantees that $\gamma_1=\gamma_2$ in $\R^n$.
    \end{proof}

    \subsubsection{Uniqueness of the potential $q$}
    \label{subsubsec: equality of q}
    
    In this section, we finally establish the uniqueness assertion in Theorem~\ref{main thm}. In fact, under the given assumptions of Theorem~\ref{main thm}, Theorem~\ref{thm: uniqueness of gamma} implies $\gamma_1=\gamma_2$ in $\R^n$. The next theorem now ensures that there also holds $q_1=q_2$ in $\Omega$.
    
    \begin{theorem}
    \label{thm: uniqueness q}
        Let $0 < s < \min(1,n/2)$, suppose $\Omega\subset \R^n$ is a domain bounded in one direction and let $W_1,W_2\subset \Omega$ be two non-disjoint measurement sets. Assume that the diffusions $\gamma_1, \gamma_2\in L^{\infty}(\R^n)$ with background deviations $m_{\gamma_1},m_{\gamma_2}\in H^{s,n/s}(\R^n)$ and potentials $q_1,q_2\in \distr(\R^n)$ satisfy
        \begin{enumerate}[(i)]
            \item\label{uniform ellipticity diffusions} $\gamma_1,\gamma_2$ are uniformly elliptic with lower bound $\gamma_0>0$,
            \item\label{continuity diffusions} $\gamma_1, \gamma_2$ are a.e. continuous in $W_1\cap W_2$,
            \item\label{integrability potentials} $q_1,q_2\in M_{\gamma_0/\delta_0,+}(H^s\to H^{-s})\cap L^p_{loc}(W_1\cap W_2)$ for $\frac{n}{2s}< p\leq \infty$
            \item\label{equal potentials in measurement sets} and $q_1|_{W_1\cap W_2}=q_2|_{W_1\cap W_2}$.
        \end{enumerate}  
        If $\Lambda_{\gamma_1,q_1}f|_{W_2}=\Lambda_{\gamma_2,q_2}f|_{W_2}$ for all $f\in C_c^{\infty}(W_1)$, then there holds $q_1=q_2$ in $\Omega$.
    \end{theorem}

    \begin{proof}
        Note that by Theorem~\ref{thm: uniqueness of gamma} we already know that the condition on the DN maps implies $\gamma_1=\gamma_2$ in $\R^n$. Now, we first show that the fractional conductivity operator $L_{\gamma}^s$ has the UCP on $H^s(\R^n)$ as long as $m_{\gamma}\in H^{s,n/s}(\R^n)$. For this purpose, assume that $V\subset\R^n$ is a nonempty, open set and $u\in H^s(\R^n)$ satisfies $L_{\gamma}^su=u=0$ in $V$. By the fractional Liouville reduction (Lemma~\ref{Lemma: fractional Liouville reduction}, \ref{liouville red identity}) and $u|_V=0$, there holds
        \[
        \begin{split}
            0&=\langle L_{\gamma}^su,\varphi\rangle \\
            &=\langle (-\Delta)^{s/2}(\gamma^{1/2}u),(-\Delta)^{s/2}(\gamma^{1/2}\varphi)\rangle_{L^2(\R^n)}-\left\langle \frac{(-\Delta)^sm_{\gamma}}{\gamma}\gamma^{1/2} u,\gamma^{1/2}\varphi \right\rangle\\
            &=\langle (-\Delta)^{s/2}(\gamma^{1/2}u),(-\Delta)^{s/2}(\gamma^{1/2}\varphi)\rangle_{L^2(\R^n)}
        \end{split}
        \]
        for any $\varphi\in C_c^{\infty}(V)$. By approximation the above identity holds for all $\varphi\in \widetilde{H}^s(V)$. By the property \ref{estimate mathcal M} of Lemma~\ref{Lemma: fractional Liouville reduction}, we can replace $\varphi\in \widetilde{H}^s(V)$ by $\psi=\gamma^{-1/2}\varphi\in \widetilde{H}^s(V)$ to see that $(-\Delta)^{s/2}(\gamma^{1/2}u)=0$ in $V$. Now, the UCP for the fractional Laplacian implies $\gamma^{1/2}u=0$ in $\R^n$. Hence, the uniform ellipticity of $\gamma$ ensures $u=0$ in $\R^n$. 

        Hence, the problem at hand satisfies the conditions in \cite[Theorem~2.6]{RZ2022unboundedFracCald} (see also \cite[Remark~4.2]{RZ2022unboundedFracCald}, Theorem~\ref{thm: well-posedness for fractional Schrödinger type equation} and Remark~\ref{remark: interior source problem}) and we obtain $q_1=q_2$ in $\Omega$.
    \end{proof}
    

    \subsection{Remarks on assumption \ref{integrability potentials} in Theorem~\ref{main thm}}
    \label{subsec: assump potent}

    Before proceeding in the next section to the construction of counterexamples to non-uniqueness, we discuss here the assumption \ref{integrability potentials} in Theorem~\ref{main thm}. More precisely, we answer here the following question:

    \begin{question}
    \label{question on assumption for potential}
        Do there exist conductivities $\gamma\colon\R^n\to \R$ such that the potentials $q_{\gamma}=-(-\Delta)^s m_{\gamma}/\gamma^{1/2}$ coming from the Liouville reduction satisfy the assumption \ref{integrability potentials} in Theorem~\ref{main thm}?
    \end{question}

    A simple example is the constant conductivity $\gamma=1$. In the next proposition we show that one can actually construct a whole class of conductivities via the obstacle problem for the fractional Laplacian, which has been studied in recent years by many authors (see e.g. \cite{SI-regularity-obstacle-problem,CSS2008OptimalReg,FernandezRosOtonCriticalDrift,JhaveriNeumayer2017HigherRegFracLap,CaffarelliRosOtonSerra2017Integro,AbatangeloRosOtonHigherRegularity} and the references therein).

    \begin{proposition}
    \label{prop: construction of potentials}
        Assume that $\varphi\in C_c^{1,1}(\R^n)$ is nonnegative and let $\rho\in C_c^{\infty}(\R^n)$ be a nonnegative mollifier. Then for any $0<s<n/4$ there exists a uniformly elliptic, smooth conductivity $\gamma\colon\R^n\to \R$ such that 
        \begin{enumerate}[(A)]
            \item\label{prop 1 constr} $m_{\gamma}\in H^{s,n/s}(\R^n)$,
            \item\label{prop 2 constr} $m_{\gamma}\geq \rho\ast\varphi\geq 0$ in $\R^n$,
            \item\label{prop 3 constr} $q_{\gamma}\in M(H^s\to H^{-s})\cap L^{\infty}(\R^n)$
            \item\label{prop 4 constr} and $q_{\gamma}\geq 0$ in $\R^n$.
        \end{enumerate}
    \end{proposition}
    \begin{proof}
        Consider the obstacle problem for the fractional Laplacian
    \begin{equation}
    \label{eq: FBP}
        \begin{array}{rl} 
    -(-\Delta)^s u &\!\!\!\geq 0 \ \text{ in } \R^n, \\
    (-\Delta)^s u &\!\!\!= 0 \ \text{ in } \{u>\varphi\}, \\
    u &\!\!\!\geq \varphi \ \text{ in } \R^n. \end{array} 
    \end{equation}
    Now, recall that by variational methods one can construct the unique solution $u\in \dot{H}^s(\R^n)$ to \eqref{eq: FBP} (see \cite[Section~3.1]{SI-regularity-obstacle-problem}). Moreover, by \cite[Corollary 3.7 and 3.9]{SI-regularity-obstacle-problem} the solution $u$ is bounded and Lipschitz continuous. Note that by the Sobolev embedding $\dot{H}^s(\R^n)\hookrightarrow L^{\frac{2n}{n-2s}}(\R^n)$, we have $u\in L^{\frac{2n}{n-2s}}(\R^n)$. Thus, using $u\in L^{\frac{2n}{n-2s}}(\R^n)\cap L^{\infty}(\R^n)$, $0<s<n/4$, the identity
    \[
        1+\frac{s}{n}=\frac{n+s}{n}=\frac{n-2s}{2n}+\frac{n+4s}{2n}
    \]
     and Young's inequality, we can conclude that 
    \begin{equation}
    \label{eq: regularity m}
        m\vcentcolon = \rho\ast u\in L^{n/s}(\R^n)\cap L^{\infty}(\R^n).
    \end{equation}
    Since the fractional Laplacian commutes with convolution and $\rho\in C_c^{\infty}(\R^n)$, we also get by the same argument 
    \begin{equation}
    \label{eq: regularity frac lap m}
        (-\Delta)^{s/2}m=(-\Delta)^{s/2}\rho\ast u\in L^{n/s}(\R^n).
    \end{equation}
    Again using that the fractional Laplacian commutes with convolution, $\varphi$ is nonnegative, $\rho$ is a nonnegative mollifier and $u\geq \varphi$ in $\R^n$, we have 
    \begin{enumerate}[(a)]
        \item\label{nonneg pot} $-(-\Delta)^s m\geq 0$ in $\R^n$
        \item and $m\geq \rho\ast \varphi\geq 0$.
    \end{enumerate}
    Taking into account \eqref{eq: regularity m}, we obtain that the function $\gamma\vcentcolon = (m+1)^2$ is uniformly elliptic and by definition $m=m_{\gamma}$. The smoothness of $\gamma$ comes from the fact that $\rho\in C_c^{\infty}(\R^n)$. Moreover, by \eqref{eq: regularity m} and \eqref{eq: regularity frac lap m}, we conclude that $m\in H^{s,n/s}(\R^n)$ and thus we can conclude that $m$ fulfills \ref{prop 1 constr} and \ref{prop 2 constr}. Furthermore, using $m\in H^{s,n/s}(\R^n)$ and the assertion \ref{potential frac Liouville reduction} of Lemma~\ref{Lemma: fractional Liouville reduction}, we deduce $q\vcentcolon = -(-\Delta)^sm/\gamma^{1/2}\in M(H^s\to H^{-s})$. This potential $q$ is by definition the potential $q_{\gamma}$ coming from the Liouville reduction. Finally, observe that $\rho\in C_c^{\infty}(\R^n)$ and Young's inequality ensures
    \[
        (-\Delta)^sm=(-\Delta)^s\rho\ast u\in L^{\infty}(\R^n).
    \]
    Thus, we see by this, the uniform ellipticity of $\gamma\in C^{\infty}(\R^n)$ and \ref{nonneg pot} that $q=-(-\Delta)^sm/\gamma^{1/2}\in L^{\infty}(\R^n)$ and $q\geq 0$. Hence, $q$ satisfies the conditions \ref{prop 3 constr}--\ref{prop 4 constr}.
    \end{proof}

    \subsection{Non-uniqueness}
    \label{sec: non-uniqueness}
    
    In this section, we construct counterexamples to uniqueness when the potentials are non-equal in the whole measurement set $W$ and hence prove Theorem~\ref{thm: non uniqueness}. Similarly, as in the articles \cite{counterexamples,RZ2022LowReg}, the construction of counterexamples relies on a PDE characterization of the equality of the DN maps. To derive such a correspondence between DN maps and a PDE for the coefficients, we need the following lemma:
    
    \begin{lemma}[Relation to fractional Schr\"odinger problem]
    \label{Auxiliary lemma}
        Let $\Omega\subset\R^n$ be an open set which is bounded in one direction, $W\subset\Omega_e$ an open set and $0<s<\min(1,n/2)$. Assume that $\gamma,\Gamma\in L^{\infty}(\R^n)$ with background deviations $m_{\gamma},m_{\Gamma}$ satisfy $\gamma(x),\Gamma(x)\geq \gamma_0>0$ and $m_{\gamma},m_{\Gamma}\in H^{s,n/s}(\R^n)$. Moreover, let $q\in M_{\gamma_0/\delta_0,+}(H^s\to H^{-s})$. If $\gamma|_{W}=\Gamma|_{W}$, then
    \begin{equation}
    \label{eq: identity DN maps}
        \langle \Lambda_{\gamma,q}f,g\rangle=\langle \Lambda_{Q_{\gamma,q}}(\Gamma^{1/2}f),(\Gamma^{1/2}g)\rangle
    \end{equation}
    holds for all $f,g\in\widetilde{H}^{s}(W)$, where the potential $Q_{\gamma,q}\in M(H^s\to H^{-s})$ is given by formula \eqref{eq: reduced potential}.
    \end{lemma}
    
    \begin{proof}
        First recall that if $u_f\in H^s(\R^n)$ is the unique solution to 
    \[
    \begin{split}
    L_{\gamma}^s u+qu&= 0\enspace\text{in}\enspace\Omega,\\
            u&= f\enspace\text{in}\enspace\Omega_e
    \end{split}
    \]
    with $f\in \widetilde{H}^s(W)$, then $\gamma^{1/2}u_f\in H^s(\R^n)$ is the unique solution to
    \[
        \begin{split}
            ((-\Delta)^s+Q_{\gamma,q})v&=0\quad\,\,\,\,\enspace\text{in}\enspace\Omega,\\
            v&=\gamma^{1/2}f\enspace\text{in}\enspace\Omega_e
        \end{split} 
    \] 
    (see~Theorem~\ref{thm: well-posedness for fractional Schrödinger type equation}, \ref{item 2 well-posedness schrödinger}). Since $\gamma|_W = \Gamma|_W$, we have $\gamma^{1/2}f=\Gamma^{1/2}f$ and therefore $\gamma^{1/2}u_f$ is the unique solution to
    \[
            \begin{split}
            ((-\Delta)^s+Q_{\gamma,q})v&=0\quad\,\,\,\,\enspace\text{in}\enspace\Omega,\\
            v&=\Gamma^{1/2}f\enspace\text{in}\enspace\Omega_e,
        \end{split}
    \] 
    which we denote by $v_{\Gamma^{1/2}f}$. Using the property \ref{liouville red identity} of Lemma~\ref{Lemma: fractional Liouville reduction} and the definition of $Q_{\gamma,q}$ via formula \eqref{eq: reduced potential}, we deduce
    \[
    \begin{split}
        \langle \Lambda_{\gamma,q}f,g\rangle&=B_{\gamma,q}(u_f,g)=B_{\gamma}(u_f,g)+\langle qu_f,g\rangle\\
        &=B_{q_{\gamma}}(\gamma^{1/2}u_f,\gamma^{1/2}g)+\left\langle \frac{q}{\gamma}(\gamma^{1/2}u_f),\gamma^{1/2}g\right\rangle\\
        &=B_{Q_{\gamma,q}}(\gamma^{1/2}u_f,\gamma^{1/2}g)=B_{Q_{\gamma,q}}(v_{\Gamma^{1/2}f},\Gamma^{1/2}g)\\
        &=\langle \Lambda_{Q_{\gamma,q}}(\Gamma^{1/2}f),(\Gamma^{1/2}g)\rangle
    \end{split}
    \]
    for all $f,g\in\widetilde{H}^s(W)$. In the last equality sign we used the definition of the DN map $\Lambda_{Q_{\gamma,q}}$ given in Theorem~\ref{thm: well-posedness for fractional Schrödinger type equation}, \ref{item 5 well-posedness schrödinger}.
    \end{proof}
    
    With this at hand, we can now give the proof of Theorem~\ref{thm: non uniqueness}:
    
    \begin{proof}[{Proof of Theorem~\ref{thm: non uniqueness}}]
        First assume that the coefficients $(\gamma_1,q_1)$ and $(\gamma_2,q)$ satisfy the regularity assumptions of Theorem~\ref{main thm}. Next, denote by $\Gamma\colon\R^n\to\R_+$ any function satisfying the following conditions
        \begin{enumerate}[(a)]
            \item $\Gamma\in L^{\infty}(\R^n)$,
            \item $\Gamma\geq \gamma_0$,
            \item $\Gamma|_W=\gamma_1|_W=\gamma_2|_W$
            \item and $m_{\Gamma}=\Gamma^{1/2}-1\in H^{s,n/s}(\R^n)$.
        \end{enumerate}
        By Lemma~\ref{Auxiliary lemma}, Theorem~\ref{thm: well-posedness for fractional Schrödinger type equation} and Theorem~\ref{thm: exterior reconstruction}, one sees that $\Lambda_{\gamma_1,q_1}f|_W=\Lambda_{\gamma_2,q_2}f|_W$ for all $f\in C_c^{\infty}(W)$ is equivalent to $\Lambda_{Q_{\gamma_1,q_1}}f|_W=\Lambda_{Q_{\gamma_2,q_2}}f|_W$ for all $f\in C_c^{\infty}(W)$ and $\gamma_1|_W=\gamma_2|_W$. Next, we claim this is equivalent to the following two assertions:
        \begin{enumerate}[(i)]
            \item $\gamma_1=\gamma_2$ in $W$,
            \item\label{item 1 equal potentials} $Q_{\gamma_1,q_1}=Q_{\gamma_2,q_2}$ in $\Omega$
            \item\label{item 2 equality in exterior set} and 
        \begin{equation}
        \label{eq: equivalence measurement set}
            (-\Delta)^sm+\frac{q_2-q_1}{\gamma_2^{1/2}}=0\enspace\text{in}\enspace W,
        \end{equation}
        where $m=m_{\gamma_1}-m_{\gamma_2}$.
        \end{enumerate}
        If $\Lambda_{Q_{\gamma_1,q_1}}f|_W=\Lambda_{Q_{\gamma_2,q_2}}f|_W$ for all $f\in C_c^{\infty}(W)$, then \cite[Theorem~2.6, Corollary~2.7]{RZ2022unboundedFracCald} ensure that $Q_{\gamma_1,q_1}=Q_{\gamma_2,q_2}$ in $\Omega$ and $W$. Next note that
        \begin{equation}
        \label{eq: some calculation}
            \begin{split}
                0&=Q_{\gamma_1,q_1}-Q_{\gamma_2,q_2}=-\frac{(-\Delta)^sm_{\gamma_1}}{\gamma_1^{1/2}}+\frac{(-\Delta)^sm_{\gamma_2}}{\gamma_2^{1/2}}+\frac{q_1}{\gamma_1}-\frac{q_2}{\gamma_2}\\
                &=-\frac{(-\Delta)^sm}{\gamma_1^{1/2}}+\left(\frac{1}{\gamma_2^{1/2}}-\frac{1}{\gamma_1^{1/2}}\right)(-\Delta)^sm_{\gamma_2}+\frac{q_1}{\gamma_1}-\frac{q_2}{\gamma_2}\\
                &=-\frac{(-\Delta)^sm}{\gamma_1^{1/2}}+\frac{m}{\gamma_1^{1/2}\gamma_2^{1/2}}(-\Delta)^sm_{\gamma_2}+\frac{q_1}{\gamma_1}-\frac{q_2}{\gamma_2},
            \end{split}
        \end{equation}
        where set $m=m_{\gamma_1}-m_{\gamma_2}$. As $\gamma_1=\gamma_2$ in $W$, the identity \eqref{eq: some calculation} reduces to the one in statement \ref{item 2 equality in exterior set}.
        Next, assume the converse namely that $\gamma_1=\gamma_2$ in $W$ and $m=m_{\gamma_1}-m_{\gamma_2}$ as well as  $Q_{\gamma_j,q_j}$ for $j=1,2$ satisfy \ref{item 1 equal potentials} and \ref{item 2 equality in exterior set}. Then for any given Dirichlet value $f\in C_c^{\infty}(W)$, the Dirichlet problems for $(-\Delta)^sv+Q_{\gamma_1,q_1}v=0$ in $\Omega$ and $(-\Delta)^sv+Q_{\gamma_2,q_2}v=0$ in $\Omega$ have the same solution $v_f^{(1)}=v_f^{(2)}$. Hence, one has \[
        \begin{split}
            B_{Q_{\gamma_1,q_1}}(v_f^{(1)},g)&=\langle (-\Delta)^{s/2}v_f^{(1)},(-\Delta)^{s/2}g\rangle+\langle Q_{\gamma_1,q_1}v_f^{(1)},g\rangle\\
            &=\langle (-\Delta)^{s/2}v_f^{(2)},(-\Delta)^{s/2}g\rangle+\langle Q_{\gamma_1,q_1}f,g\rangle\\
            &=\langle (-\Delta)^{s/2}v_f^{(2)},(-\Delta)^{s/2}g\rangle+\langle Q_{\gamma_2,q_2}f,g\rangle\\
            &=\langle (-\Delta)^{s/2}v_f^{(2)},(-\Delta)^{s/2}g\rangle+\langle Q_{\gamma_2,q_2}v_f^{(2)},g\rangle
        \end{split}
        \]
        for any $g\in C_c^{\infty}(W)$, but this is nothing else than $\Lambda_{Q_{\gamma_1,q_1}}f|_W=\Lambda_{Q_{\gamma_1,q_1}}f|_W$.

    Next, choose $\gamma_2=1$ and $q_2=0$ and assume that $(\gamma_1,q_1)$ satisfies the assumptions of Theorem~\ref{main thm}. This implies that there holds $\Lambda_{\gamma_1,q_1}f|_W=\Lambda_{1,0}f|_W$ for all $f\in C_c^{\infty}(W)$ if and only if we have
    \begin{enumerate}[(I)]
            \item\label{item 2 measurement set 2} $\gamma_1=1$ on $W$,
            \item\label{item 1 equal potentials 2} $Q_{\gamma_1,q_1}=0$ in $\Omega$
            \item\label{item 3 equality in exterior set} and $(-\Delta)^sm_{\gamma_1}=q_1\enspace\text{in}\enspace W$.
    \end{enumerate}
    Therefore, if we define $q_1$ via 
    \begin{equation}
    \label{eq: specification of potential}
        q_1=\gamma_1^{1/2}(-\Delta)^sm_{\gamma_1}\enspace\text{in}\enspace\R^n
    \end{equation}
    for a given sufficiently regular function $\gamma_1\colon\R^n\to\R_+$ with $\gamma_1|_W=1$, then the conditions \ref{item 2 measurement set 2}, \ref{item 1 equal potentials 2} and \ref{item 3 equality in exterior set} are satisfied. Hence, the remaining task is to select $\gamma_1$ in such a way that the required regularity properties of Theorem~\ref{main thm} are met. We construct $m_{\gamma_1}\in H^{s,n/s}(\R^n)\cap H^s(\R^n)$ as follows: First, choose open sets $\Omega',\omega\subset\R^n$ satisfying $\Omega'\Subset\Omega$ and $\omega\Subset \Omega\setminus\overline{\Omega'}$. Next, let us fix some $\epsilon>0$ such that $\Omega'_{5\epsilon}, \omega_{5\epsilon}, \Omega_e$ are disjoint. Here and in the rest of the proof, we denote by $A_{\delta}$ the open $\delta$-neighbor-hood of the set $A\subset\R^n$. 
    Now, choose any nonnegative cut-off function $\eta\in C_c^{\infty}(\omega_{3\epsilon})$ satisfying $\eta|_{\omega}=1$. We define $\widetilde{m}\in H^s(\R^n)$ as the unique solution to
    \begin{equation}
    \label{eq: PDE in extended domain}
        (-\Delta)^s\widetilde{m}=0\quad \text{in}\quad \Omega'_{2\epsilon},\quad \widetilde{m}=\eta\quad\text{in}\quad \R^n\setminus\overline{\Omega'}_{2\epsilon}.
    \end{equation}
    Since $\eta\geq 0$, the maximum principle for the fractional Laplacian shows $\widetilde{m}\geq 0$ (cf.~\cite[Proposition~4.1]{RosOton16-NonlocEllipticSurvey}). Proceeding as in \cite[Proof of Theorem 1.6]{RZ2022LowReg} one can show that
    \[
        m_{\gamma_1}\vcentcolon =C_{\epsilon}\rho_{\epsilon}\ast\widetilde{m}\in H^s(\R^n)\quad\text{with}\quad C_{\epsilon}\vcentcolon=\frac{\epsilon^{n/2}}{2|B_1|^{1/2}\|\rho\|_{L^{\infty}(\R^n)}^{1/2}\|\widetilde{m}\|_{L^2(\R^n)}},
    \]
    where $\rho_{\epsilon}\in C_c^{\infty}(\R^n)$ is the standard mollifier of width $\epsilon$, solves
    \[
        (-\Delta)^sm=0\quad \text{in}\quad \Omega',\quad m=m_{\gamma_1}\quad\text{in}\quad \Omega'_e.
    \]
    Furthermore, $m_{\gamma_1}$ has the following properties
\begin{enumerate}[(A)]
    \item\label{item 2 m} $m_{\gamma_1}\in L^{\infty}(\R^n)$ with $\|m_{\gamma_1}\|_{L^{\infty}(\R^n)}\leq 1/2$ and $m_{\gamma_1}\geq 0$,
    \item\label{item 3 m} $m_{\gamma_1}\in H^{s}(\R^n)\cap H^{s,n/s}(\R^n)$
    \item\label{item 4 m} and $\supp(m_{\gamma_1})\subset \Omega_e$.
    \end{enumerate}
    Now, we define $\gamma_1\in L^{\infty}(\R^n)$ via $\gamma_1=(m_{\gamma_1}+1)^2\geq 1$. Therefore, $\gamma_1$ satisfies all required properties and even belongs to $C^{\infty}_b(\R^n)$, since $m_{\gamma_1}$ is defined via mollification of a $L^2$ function. Using a similar calculation as for Lemma~ \ref{Lemma: fractional Liouville reduction}, \ref{potential frac Liouville reduction}, we have $q_1\in M(H^s\to H^{-s})$ and by scaling of $m_{\gamma_1}$ we can make the norm $\|q_1\|_s$ as small as we want. In particular, this allows to guarantee $q_1\in M_{\gamma_0/\delta_0}(H^s\to H^{-s})$ with $\gamma_0=1$. Note that we cannot have $q_1|_W=0$ as then the UCP implies $m_{\gamma_1}=0$. Hence, we can conclude the proof.
    \end{proof}

    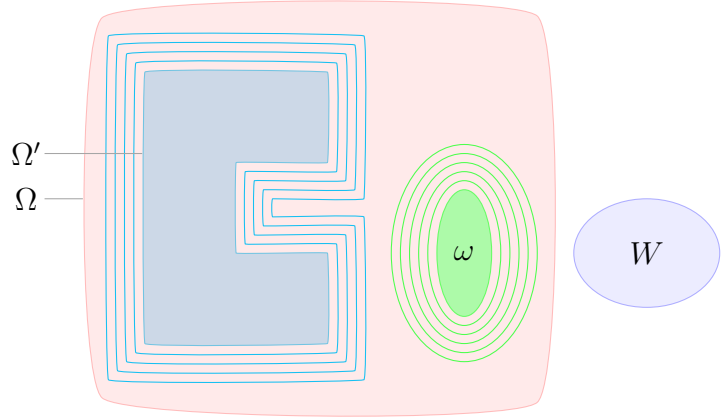
\begin{figure}[!ht]
    \centering
    \begin{tikzpicture}[scale=1.2]
    \draw[color=cyan!70, fill=cyan!30] plot [smooth cycle, tension=0.05] coordinates {(0,0) (0,3) (2,3) (2,2)  (1,2) (1,1) (2,1) (2,0)};
    \filldraw[color=red!60, fill=red!20, opacity = 0.4] plot [smooth cycle, tension=0.2] coordinates {(-0.47,-0.6) (-0.47,3.6) (4.3,3.62) (4.3,-0.65) };
    \draw[color=cyan!70] plot [smooth cycle, tension=0.05] coordinates { (-0.1,-0.1) (-0.1,3.1) (2.1,3.1) (2.1,1.9) (1.1,1.9) (1.1,1.1) (2.1,1.1) (2.1,-0.1)};
    \draw[color=cyan!70] plot [smooth cycle, tension=0.05] coordinates {  (-0.2,-0.2)  (-0.2,3.2)  (2.2,3.2) (2.2,1.8)  (1.2,1.8) (1.2,1.2) (2.2,1.2) (2.2,-0.2)};
    \draw[color=cyan!70] plot [smooth cycle, tension=0.05] coordinates {   (-0.3,-0.3) (-0.3,3.3) (2.3,3.3) (2.3,1.7) (1.3,1.7) (1.3,1.3)  (2.3,1.3) (2.3,-0.3)};
    \draw[color=cyan!70] plot [smooth cycle, tension=0.05] coordinates {   (-0.4,-0.4) (-0.4,3.4) (2.4,3.4) (2.4,1.6) (1.4,1.6) (1.4,1.4) (2.4,1.4) (2.4,-0.4)};
    \draw[color=gray!70] (-0.03,2.1)--(-1.1,2.1);
    \node at (-1.3,2.1) {$\raisebox{-.35\baselineskip}{\ensuremath{\Omega'}}$};
  \draw[color=gray!70] (-0.68,1.6)--(-1.1,1.6);
    \node at (-1.3,1.6) {$\raisebox{-.35\baselineskip}{\ensuremath{\Omega}}$};
    \filldraw[color=blue!60, fill=blue!15, opacity=0.5] (5.5,1) ellipse (0.8 and 0.6);
    \node at (5.5,1) {$\raisebox{-.35\baselineskip}{\ensuremath{W}}$};
    \filldraw[color=green!70, fill=green!40, opacity=0.8] (3.5,1) ellipse (0.3 and 0.7);
    \draw[color=green!70] (3.5,1) ellipse (0.4 and 0.8);
    \draw[color=green!70] (3.5,1) ellipse (0.5 and 0.9);
    \draw[color=green!70] (3.5,1) ellipse (0.6 and 1);
    \draw[color=green!70] (3.5,1) ellipse (0.7 and 1.1);
    \draw[color=green!70] (3.5,1) ellipse (0.8 and 1.2);
     \node at (3.5,1) {$\raisebox{-.35\baselineskip}{\ensuremath{\omega}}$};
\end{tikzpicture}\label{fig: Geometric setting 2}
    \caption{A graphical illustration of the sets used in the proof of Theorem~\ref{thm: non uniqueness}. }
\end{figure}

\medskip 
	
	\noindent\textbf{Acknowledgment.} The author thanks the anonymous referees and the editor for the careful reading and the helpful comments.

\bibliographystyle{alpha}

\end{document}